\documentclass[12pt]{amsart}
\usepackage{amssymb}
\usepackage[margin=2.5cm]{geometry}

\newtheorem{thm}{Theorem}[section]
\newtheorem{cor}[thm]{Corollary}
\newtheorem{lem}[thm]{Lemma}
\newtheorem{prop}[thm]{Proposition}
\theoremstyle{definition}
\newtheorem{defn}[thm]{Definition}

\newtheorem*{claim*}{Claim}
\theoremstyle{remark}
\newtheorem{rem}[thm]{Remark}
\newtheorem{example}[thm]{Example}
\numberwithin{equation}{section}
\newcommand{\norm}[1]{\left\Vert#1\right\Vert}
\newcommand{\bignorm}[1]{\bigl\Vert#1\bigr\Vert}
\newcommand{\abs}[1]{\left\vert#1\right\vert}

\newcommand{\BL}{\operatorname{\mathbf{BL}}}
\newcommand{\BLhat}{\operatorname{\widehat{\mathbf{BL}}}}
\newcommand{\N}{\mathbb{N}}
\newcommand{\R}{\mathbb{R}}
\newcommand{\C}{\mathbb{C}}
\newcommand{\Q}{\mathbb{Q}}
\newcommand{\T}{\mathbb{T}}
\newcommand{\Z}{\mathbb{Z}}
\renewcommand{\O}{\mathbb{O}}
\newcommand{\calO}{\mathcal{O}}
\newcommand{\calF}{\mathcal{F}}

\newcommand{\calE}{\mathcal{E}}
\newcommand{\bsG}{\boldsymbol{G}}

\newcommand{\bsGh}{\boldsymbol{\hat{G}}}
\newcommand{\bsp}{\boldsymbol{p}}
\newcommand{\bspp}{\boldsymbol{p'}}
\newcommand{\bss}{\boldsymbol{\sigma}}
\newcommand{\bsds}{\boldsymbol{\dot{\sigma}}}
\newcommand{\bssf}{\boldsymbol{\sigma^4}}
\newcommand{\bsso}{\boldsymbol{\sigma^1}}
\newcommand{\bsstw}{\boldsymbol{\sigma^2}}
\newcommand{\bssth}{\boldsymbol{\sigma^3}}

\newcommand{\bsts}{\boldsymbol{\tilde{\sigma}}}
\newcommand{\bshs}{\boldsymbol{\hat{\sigma}}}

\newcommand{\bsth}{\boldsymbol{\theta}}

\newcommand{\bsq}{\boldsymbol{q}}
\newcommand{\bsr}{\boldsymbol{\rho}}
\newcommand{\bsA}{\boldsymbol{A}}
\newcommand\g{{\operatorname{g}}}
\newcommand\BLg{{\operatorname{BL_\g}}}
\newcommand\s{{\operatorname{s}}}
\newcommand\BLs{{\operatorname{BL_\s}}}

\newcommand{\one}{^{(1)}}
\newcommand{\two}{^{(2)}}

\newcommand{\Alpha}{\mathrm{A}}
\newcommand{\Beta}{\mathrm{B}}
\newcommand{\Fourier}{\mathcal{F}}

\newcommand{\lnorm}{\left\Vert}
\newcommand{\rnorm}{\right\Vert}

\newcommand{\lip}{\left\langle}
\newcommand{\rip}{\right\rangle}

\newcommand{\labs}{\left\vert}
\newcommand{\rabs}{\right\vert}
\newcommand{\bigglabs}{\biggl\vert}
\newcommand{\biggrabs}{\biggr\vert}

\newcommand{\lpar}{\left(}
\newcommand{\rpar}{\right)}

\newcommand{\wrt}[1]{\ignorespaces\,\mathrm{d}#1}
\newcommand{\fn}{\ignorespaces\,}

\newcommand{\afterbar}{\bar{\phantom{x}}}

\newcommand{\supp}{\operatorname{supp}}
\newcommand{\Hom}{\operatorname{Hom}}
\newcommand{\Ann}{\operatorname{Ann}}
\newcommand{\ds}{\displaystyle}

\begin{document}

\title[Brascamp--Lieb inequalities on LCA groups]%
{Brascamp--Lieb inequalities on locally compact abelian groups}%
\author{Jonathan Bennett}
\address{School of Mathematics, University of Birmingham, Edgbaston, Birmingham, B15 2TT, United Kingdom}
\email{j.bennett@bham.ac.uk}
\author{Michael G. Cowling}
\address{School of Mathematics, University of New South Wales, UNSW Sydney NSW 2052, Australia}
\email{m.cowling@unsw.edu.au}


\begin{abstract}
We establish a structure theorem for the Brascamp--Lieb constant formulated in the general setting of locally compact abelian groups. 
This extends and unifies the finiteness characterisations previously known for euclidean spaces and for finitely generated groups and their duals. We place particular emphasis on Fourier invariance throughout, reflecting
the fundamental Fourier invariance of Brascamp--Lieb multilinear forms in this context. 
\end{abstract}
\thanks{The first author was supported by the Engineering and Physical Sciences Research Council through grant EP/W032880/1, and the second by the Australian Research Council through grant DP220100285.
Both thank Eunhee Jeong for numerous discussions on the subject of this paper.
In particular the second author thanks her for explaining her paper [8] with the first author
in detail at a MATRIX meeting in March 2022, and thanks MATRIX for helping to fund her visit to this meeting. }
\maketitle


\section{Introduction and main results}\label{sec:intro}
The Brascamp--Lieb inequalities (also known as H\"older--Brascamp--Lieb inequalities), introduced in \cite{BL76}, are of considerable interest across mathematics; see, for example, \cite{BB21, Zh22} for an illustration of their influence, particularly in Fourier analysis. They concern Lebesgue space bounds for rather broad classes of multilinear forms that arise frequently in a range of settings, from harmonic analysis and dispersive PDE, to additive combinatorics, convex geometry and theoretical computer science.
Originally they were phrased in the context of euclidean spaces, but recently versions have been considered in the context of discrete abelian groups and of finite-dimensional compact abelian groups, as well as in the context of compact homogeneous spaces and the Heisenberg group.
The main aim of this paper is  to describe the Brascamp--Lieb inequalities on locally compact abelian groups, illustrating the fundamental role that Fourier duality plays in these.
We shall see more about the history of these inequalities after we state what they are and outline our main results.

Suppose that $G$ is a locally compact (Hausdorff topological) group, with identity $e$, usually written multiplicatively; integrals over $G$ are relative to its left-translation-invariant Haar measure.
We \emph{always} suppose that homomorphisms of locally compact groups are continuous.

In this paper we focus on the case where $G$ is an LCA group, that is, a locally compact \emph{abelian} group.
This case is simpler than the general case for two reasons: first, all subgroups are normal and so all quotients $G/H$ are groups rather than homogeneous spaces, and second, the Haar measure is also invariant under right translations.
In a future paper we propose to deal with the general case.
However, some of our results are stated for the general case, as the proofs do not rely on commutativity; we always use expressions such as ``not necessarily abelian'' to clarify that our groups may be nonabelian.

Every LCA group $G$ has a dual group $\hat G$, and given a suitable function $f$ on $G$, its Fourier transform $\hat f$ is defined on $\hat{G}$.
The Haar measure on the LCA group $\hat{G}$ may be normalised such that the Plancherel and inversion formulas hold, for ``suitable'' functions on $G$.
We write $S(G)$ for a suitable space of test functions on $G$.
We make these notions precise later.

\begin{defn}  \label{def:BL-datum}
A Brascamp--Lieb (homomorphism) datum is a triple $(G, \bss, \bsp)$, where
$\bss$ is a $J$-tuple of homomorphisms $\sigma_j : G \to G_j$ of LCA groups,
and $\bsp$ is a $J$-tuple of indices $(p_1, \dots, p_J)$, each in $[1,+\infty]$.
\end{defn}

The index $J$ is a positive integer, which is not particularly important; we omit it in definitions as well as in sums and products.
When $1 \leq p \leq \infty$, the dual index is written $p'$.

\begin{defn}\label{def:BL-ineq-homo}
A Brascamp--Lieb inequality is an inequality of the form
\begin{equation}\label{usual}
\labs \int_{G} \prod_j (f_j \circ \sigma_j)(x) \wrt{x} \rabs
\leq C \prod_{j} \lnorm f_j \rnorm_{L^{p_j}(G_j)}
\qquad\forall f_j \in S(G_j) ,
\end{equation}
and a (Fourier) dual Brascamp--Lieb inequality is an inequality of the form
\begin{equation}\label{unusual}
\labs \int_{G} \prod_j (f_j \circ \sigma_j)(x) \wrt{x} \rabs
\leq C \prod_{j} \lnorm \hat f_j \rnorm_{L^{p'_j}(G_j)}
\qquad\forall f_j \in S(G_j).
\end{equation}
We assume that the tensor product function in the integral on the left hand sides of these inequalities is integrable.
The smallest possible values of $C$ in \eqref{usual} and \eqref{unusual}, that we denote by $\BL(G, \bss, \bsp)$ and  $\BLhat(G, \bss, \bspp)$, are called the \textit{Brascamp--Lieb constant} and \textit{dual Brascamp--Lieb constant}, repectively. These constants are taken to be $\infty$ if no such inequality holds. 
\end{defn}
It is clear that the Brascamp--Lieb constants $\BL(G, \bss, \bsp)$ and  $\BLhat(G, \bss, \bspp)$ depend on the normalisations of the Haar measures on $G$ and the $G_j$ that implicitly appear in \eqref{usual} and \eqref{unusual}.  
However, finiteness of the constants does not, since changing the Haar measure by a factor changes integrals by the same factor and $L^p$-norms by a power of the factor.
In the classical case of vector groups, where it is customary to take vector spaces with an inner product, this point does not arise as the measures are naturally the Lebesgue measures determined by the associated inner products.
Similarly, it is usually assumed that compact groups have total Haar measure $1$ and points in discrete groups have Haar measure $1$.
However, finite groups are both compact and discrete making these assumptions contradictory, so care is needed.
We clarify this issue later.

We point out that our definition of the Brascamp--Lieb constant $\BL(G, \bss, \bsp)$ differs slightly from that in several previous works, beginning with \cite{BCCT08}. This difference, which is entirely superficial, stems from our emphasis on multilinearity and Fourier invariance, which necessitates the inclusion of signed functions $f_j$; see \cite{BJ} for further clarification. The systematic study of the dual constant $\BLhat(G, \bss, \bspp)$ is more recent, originating in \cite{BBBCF} in the euclidean setting (see also \cite{BJ}), although its motivation is evident much earlier in the literature, notably in \cite{Ball89}.

Of course the fundamental questions here concern the \textit{finiteness} of these two constants for a given datum $(G, \bss, \bsp)$, and \textit{how they are related to each other}. The main purpose of this paper is to provide answers to these questions; see the forthcoming Theorem \ref{thm:intro-thm-2} and its corollary. As we shall see, in our Fourier-invariant setting of LCA groups, the inequalities \eqref{usual} and \eqref{unusual} are naturally considered in tandem.

Answers to these questions, particularly those concerning the finiteness of $\BL(G, \bss, \bsp)$, are well known in some important examples, where they may be formulated as ``rank conditions'' on compactly generated subgroups $H\leq G$ and their images under the homomorphisms $\sigma_j$. We recall that every compactly generated LCA group $H$ is of the form $\R^a \times \Z^c \times K$, where $a, c \in \N$ and $K$ is a compact group.
We write $\gamma(H)$ for the integer $a+c$; this is the smallest power $\gamma$ for which a growth estimate $|U^n| \lesssim_U n^\gamma$ holds for all compact subsets $U$ of $G$.
We call the \emph{rank condition for $(G,\bss,\bsp)$} the requirement that
\begin{equation}\label{eq:rank-condition-1}
  \gamma(H) \leq \sum_{j} \gamma(\sigma_j(H)) /  p_j
\end{equation}
for all compactly generated closed subgroups $H$ of $G$.

For \emph{vector groups} $G$, the first general finiteness results were found by Barthe \cite{Barthe} and Carlen, Lieb and Loss \cite{CLL04}, who treated (independently) the case where each $G_j$ is a copy of $\R$.
Then Carbery, Christ,  Tao and the first author \cite{BCCT08} found necessary and sufficient conditions for the finiteness of $\BL(G,\bss,\bsp)$ in general, which we call the BCCT conditions: namely, first that \eqref{eq:rank-condition-1} holds, and second that the homogeneity condition $\gamma(G) = \sum_{j} \gamma(\sigma_j(G)) /  p_j$ holds.
For these vector groups it is also known how to compute the constants when they are finite, at least in principle, thanks to a fundamental result of Lieb \cite{Lieb}. Lieb's theorem states that the Brascamp--Lieb constant is equal to the \textit{gaussian Brascamp--Lieb constant}
$$
\BLg(G, \bss, \bsp):=\ds\sup_{\bsA}\frac{ \prod_j \det(A_j^* A_j)^{1/2p_j} }
{ \det\lpar \sum_j \sigma_j^* A^*_j A_j \sigma_j /p_j \rpar^{1/2}  },
$$
obtained by testing \eqref{usual} on centred gaussian inputs $f_j$. Here $\bsA := (A_1, \dots, A_J)$ runs over $\prod_j \mathrm{GL}(\dim(G_j), \R)$.

For finitely generated \emph{discrete groups} $G$ it was shown by Carbery, Christ, Tao and the first author in \cite{BCCT10} that the finiteness of $\BL(G,\bss,\bsp)$ is equivalent to the rank condition \eqref{eq:rank-condition-1}, and a discrete analogue of Lieb's theorem was later obtained by Christ \cite{C13}, allowing the constant to be computed by testing on indicator functions of finite subgroups of $G$. As a result, the Brascamp--Lieb constant is equal to the \textit{subgroup Brascamp--Lieb constant}
$$
\BLs(G, \bss, \bsp):=\ds\sup_{H \in \calF(G)} \frac{\norm{1_H}_1} {\prod_{j} \norm{1_{\sigma_j(H)}}_{p_j}},
$$
where $\calF(G)$ is the collection of all finite subgroups of $G$.
Jeong and the first author reached similar conclusions in \cite{BJ} for the duals of such discrete groups.
We refer to Section \ref{sec:rank-cond} for a more detailed description of these prior results.

Returning to the general situation, as in the vector group case \cite{BCCT08} it will be convenient to exclude some degenerate situations where both the Brascamp--Lieb and dual Brascamp--Lieb constants are easily seen to be infinite. 
As we shall see, both constants are infinite unless the vector homomorphism $\bss$ is proper, that is, $\bss^{-1}(K)$ is compact for all compact sets $K$, and in particular unless $N := \cap_j \ker(\sigma_j)$ is compact and all the $\sigma_j$ are open (see Lemma \ref{lem:bss-must-be-proper}). 
When $N$ is compact, the homomorphisms $\sigma_j$ induce homomorphisms $\dot\sigma_j: G/N \to G_j$, and when the Haar measure on $G/N$ is appropriately normalised,
\[
\BL(G, \bss, \bsp) = \BL(G/N, \dot{\bss}, \bsp)
\quad\text{and}\quad
\BLhat(G, \bss, \bspp) = \BLhat(G/N, \dot{\bss}, \bspp).
\]
Thus we may also assume also that the common kernel $N$ is trivial.
Likewise, if $\tilde\sigma_j$ denotes $\sigma_j$ with codomain $\sigma_j(G)$ rather than $G_j$, and the Haar measure on the open subgroup $\sigma_j(G)$ of $G_j$ is appropriately normalised, then
\[
\BL(G, \bsts, \bsp) = \BL(G, {\bss}, \bsp)
\quad\text{and}\quad
\BLhat(G, \bsts, \bspp) = \BLhat(G, {\bss}, \bspp).
\]
Thus there is no loss of generality in also assuming that each $\sigma_j$ is surjective and open, and hence is an isomorphism of $G_j$ and $G/\ker(\sigma_j)$.
We shall refer to a  Brascamp--Lieb datum as \emph{nondegenerate} when these conditions all hold. For the purposes of studying Brascamp--Lieb data associated with finite constants, there is no loss of generality in assuming nondegeneracy, and we shall do this in the rest of this introduction.

For the purposes of \textit{relating} the Brascamp--Lieb constants $\BL(G, \bss, \bsp)$ and  $\BLhat(G, \bss, \bspp)$ it will be natural to work with what we shall refer to as \textit{canonical Brascamp--Lieb data}. 
\begin{defn}\label{def:canBL-datum}
A canonical Brascamp--Lieb datum is a Brascamp--Lieb datum $(G, \bss, \bsp)$ such that $G$ is a subgroup of the product group $\bsG : = G_1\times\cdots\times G_J$ and $\sigma_j$ is the canonical projection of $G$ onto $G_j$ for each $j$.
\end{defn}
The first thing to point out is that every Brascamp--Lieb datum is (isomorphism) equivalent to a canonical Brascamp--Lieb datum: the vector homomorphism $\bss$ identifies $G$ with a subgroup $\bss(G)$ of $\bsG$, upon which the $\sigma_j$ become the projections onto the factors $G_j$. 
We clarify that Brascamp--Lieb data $(G, \bss, \bsp)$ and $(\tilde{G}, \tilde{\bss}, \tilde{\bsp})$ are equivalent here if $\bsp=\tilde{\bsp}$ and if there are isomorphisms $\phi:G\rightarrow\tilde{G}$ and $\phi_j:G_j\rightarrow\tilde{G}_j$ such that $\tilde{\sigma}_j=\phi_j\circ\sigma_j\circ\phi^{-1}$ for each $j$. If the Haar measures involved are chosen suitably (specifically, the Haar measures on $\tilde{G}$ and $\tilde{G}_j$ are chosen to be  the pushforwards of the Haar measures on $G$ and $G_j$ respectively), it follows (largely tautologically) that 
\begin{equation}\label{cheat}
\BL(G, \bss, \bsp)=\BL(\tilde{G}, \tilde{\bss}, \tilde{\bsp}).
\end{equation} 
We caution that the natural (Lebesgue) measures in the vector group case only obey these relations if the isomorphisms $\phi$ and $\phi_j$ are isometries, meaning that a constant factor involving determinants appears in \eqref{cheat} in general; see \cite{BCCT08}.

Evidently, if $(G, \bss, \bsp)$ is canonical, then
\[
\int_{G} \prod_j (f_j \circ \sigma_j)(x) \wrt{x} = \int_{G} (f_1 \otimes \dots \otimes f_J)(x) \wrt{x} ,
\]
which interacts naturally with the Fourier transform via the Poisson summation formula
\begin{equation}\label{posum}
\int_{G} (f_1 \otimes \dots \otimes f_J)(x) \wrt x
= \int_{G^{\perp}} (\hat f_1 \otimes \dots \otimes \hat f_J)(y) \wrt y 
\end{equation}
for LCA groups; here 
$G^\perp$ is the annihilator of $G$ in the dual product group $\bsGh = \hat{G}_1 \times \dots \times \hat{G}_J$.
For canonical data it follows quickly from the definitions that 
\begin{equation}\label{switch}
\BLhat(G, \bss,\bspp) = \BL(G^\perp, \bshs,\bspp),
\end{equation}
where $\hat\sigma_j$ denotes the canonical projection from $G^\perp$ to $\hat{G}_j$. This identifies $(G^\perp, \bshs,\bspp)$ as a natural dual form of the datum $(G, \bss,\bsp)$. Correspondingly we define the \textit{dual rank condition} for $(G, \bss,\bsp)$ to be the rank condition  \eqref{eq:rank-condition-1} for $(G^\perp, \bshs, \bspp)$. As we shall see, the only seemingly essential role of canonical data in what follows is to allow us access to the Poisson summation formula \eqref{posum}. Indeed we shall work with general homomorphism data whenever \eqref{posum} plays no role.

The relationship between the Brascamp--Lieb constants $\BL(G, \bss, \bsp)$ and $\BLhat(G, \bss, \bspp)$ (or, equivalently, $\BL(G^\perp, \bshs,\bspp)$) is known whenever one has a Lieb-type theorem for both $(G, \bss, \bsp)$ and its dual --- so far, for  vector groups, and finitely generated groups and their duals.
For vector groups it was shown in \cite{BBBCF} that
\begin{equation}\label{fid}
\BL(G, \bss, \bsp) = \prod_{j} \lpar \frac{p_j^{1/p_j}}{(p_j')^{1/p_j'}}\rpar^{\dim(G_j)} \BLhat(G, \bss, \bspp).
\end{equation}
In particular it follows that a vector Brascamp--Lieb datum $(G, \bss,\bsp)$ has finite constant if and only if its dual does, meaning that the BCCT conditions are self-dual.
Similarly, for finitely generated $G$ it was shown is \cite{BJ} that
\begin{equation}\label{fidd}
\BL(G, \bss, \bsp) =\BLhat(G, \bss, \bspp).
\end{equation}

Before stating our main theorem, we recall that every LCA group $G$ has a unique maximal closed connected subgroup $G_{c}$ and a unique maximal subgroup $G_{b}$ of bounded elements, that is, elements $x$ such that $\{x^n :n \in \Z\}$ is relatively compact in $G$.
With these definitions, 
\[
\{e\} \subseteq G_{b} \cap G_{c} \subseteq G_{c} \subseteq G_{c}G_{b} \subseteq G.
\]
Further, $G^{1} := G_{b} \cap G_{c}$ is compact and has a discrete torsion-free dual; $G^{2} := G_{c}/(G_{b} \cap G_{c})$ is a self-dual vector group (isomorphic to $\R^a$ for some integer $a$); $G^{3} := (G_{b} G_{c}) /G_{c}$ is a totally disconnected group all of whose elements are bounded, and the dual of $G^{3}$ is of the same type; and $G^{4} := G/(G_{b} G_{c})$ is a discrete torsion-free group with a compact dual group.
If $G$ is an \emph{elementary LCA group}, that is, $G$ is compactly generated and has a compactly generated dual, then $G$ is of the form $\R^a \times \T^b \times \Z^c \times F$, where $a, b, c \in \N$ and $F$ is a finite group. 
In this case, $G^1 \simeq \T^b$, $G^2 \simeq \R^a$, $G^3 \simeq F$ and $G^4 \simeq \Z^c$. 

If $\sigma_j: G \to G_j$ is an open projection of LCA groups, then by a combination of restriction and factoring, $\sigma_j$ induces open projections $\sigma^{1}_j: G^{1} \to G^{1}_j$, $\sigma^{2}_j: G^{2} \to G^{2}_j$, $\sigma^{3}_j: G^{3} \to G^{3}_j$, and $\sigma^{4}_j: G^{4} \to G^{4}_j$.

Our main theorem states that the Brascamp--Lieb and dual Brascamp--Lieb constants ``factor through'' this structure, and tells us how to compute each of the resulting factors.
\begin{thm}\label{thm:intro-thm-2}
Suppose that $(G,\bss,\bsp)$ is a nondegenerate canonical Brascamp--Lieb datum.
Then when the Haar measures of $G^1$, $G^2$, $G^3$ and $G^4$ are suitably normalised, 
\[
\BL(G, \bss, \bsp)
= \BL(G^1, \bss^1, \bsp)  \BL(G^2, \bss^2, \bsp) \BL(G^3, \bss^3, \bsp) \BL(G^4, \bss^4, \bsp) .
\]
Further, 
\[\BL(G^1, \bss^1, \bsp) = 1\]
if the dual rank condition holds for $(G^1, \bss^1,\bsp)$, and is infinite otherwise;
\[
\BL(G^2, \bss^2, \bsp) = \ds\sup_{\bsA}\frac{ \prod_j \det(A_j^* A_j)^{1/2p_j} }
{ \det\lpar \sum_j (\sigma^{2}_j)^* A^*_j A_j \sigma^2_j /p_j \rpar^{1/2}  } ,
\]
where $\bsA := (A_1, \dots, A_J)$ runs over $\prod_j \mathrm{GL}(\dim(G^2_j), \R)$. 
This constant is finite if and only if the BCCT conditions hold; 
\[
\BL(G^3, \bss^3, \bsp) =  \ds\sup_{H \in \calO(G)} \frac{\norm{1_H}_1} {\prod_{j} \norm{1_{\sigma_j(H)}}_{p_j}}
, \]
where $\calO(G)$ is the collection of all compact open subgroups of $G$; 
\[
\BL(G^4, \bss^4, \bsp) = 1
\]
if the rank condition holds for $(G^4, \bss^4,\bsp)$, and is infinite otherwise.

Analogous results holds for $\BLhat(G, \bss, \bspp)$.
\end{thm}
This result, coupled with some Fourier duality arguments, leads to the following corollary.
\begin{cor}\label{cor:BJ-conjecture}
Let $(G, \bss,\bsp)$ be a nondegenerate canonical Brascamp--Lieb datum.
Then
\[
\BL(G, \bss, \bsp) = \prod_j \lpar \frac{p_j^{1/p_j}}{(p_j')^{1/p_j'}}  \rpar^{\dim(G^2_j)}  \BLhat(G, \bss, \bspp).
\]
\end{cor}
This generalisation of both \eqref{fid} and \eqref{fidd} to the setting of an arbitrary LCA group was conjectured in \cite{BJ}. By \eqref{switch} it is equivalent to the identity
\begin{equation}\label{du}
\BL(G, \bss, \bsp) = \prod_j \lpar \frac{p_j^{1/p_j}}{(p_j')^{1/p_j'}}  \rpar^{\dim(G^2_j)}  \BL(G^\perp, \bshs,\bspp),
\end{equation}
further clarifying the sense in which the data $(G, \bss, \bsp)$ and $(G^\perp, \bshs,\bspp)$ are dual. Evidently, this allows one (canonical) Brascamp--Lieb datum to be switched for another; see \cite{BBBCF,BJ} for more on this perspective.
Our proof consists of showing that $\BL(G^1, \bss, \bsp) = \BL((\hat{G})^4, \bss, \bspp)$, $\BL(G^4, \bss, \bsp) = \BL((\hat{G})^1, \bss, \bspp)$ and $\BL(G^3, \bss, \bsp) = \BL((\hat{G})^3, \bss, \bspp)$, reducing matters to the vector case \eqref{fid} applied to $G^2$.

\begin{example}[Young's convolution inequality]
The identification of the sharp constant in Young's convolution inequality
in the setting of general LCA groups \cite{Beckner} is of course a particular case of Theorem \ref{thm:intro-thm-2}. It also serves to illustrate the significance of Corollary \ref{cor:BJ-conjecture}. In the $m$-linear formulation of Young's inequality, $G_1, \hdots,G_m$ are a common group, $G_0$ say,
\[
G=\{(x_1,\hdots,x_m)\in G_0\times\cdots\times G_0:x_1\cdots x_m=e\},
\] 
and $\sum_{j=1}^m {1}/{p_j}=m-1$. In this case
it was shown in \cite{Beckner} that
\[
\BL(G, \bss, \bsp) = \prod_{j=1}^m \lpar \frac{p_j^{1/p_j}}{(p_j')^{1/p_j'}}  \rpar^{\dim(G^2_0)}.
\]
This follows quickly from \eqref{du} since
\[
G^\perp=\{\xi_0\otimes\cdots\otimes\xi_0:\xi_0\in\widehat{G}_0\},
\] 
and  $\BL(G^\perp, \bshs,\bspp)=1$, being the $m$-linear H\"older constant. 
\end{example}
In this paper, we put more detailed flesh on the bones that we have outlined. 
The Fourier invariance of the conditions in Theorem \ref{thm:intro-thm-2} suggests that we should look for Fourier invariance throughout our investigations, and we shall do this.
There are two main techniques that we are going to use.
First we use the structure theory of LCA groups to reduce matters to four \emph{basic types} of LCA groups: compact connected groups, vector spaces, totally disconnected topological torsion groups, and torsion free discrete groups.
Second, there are \emph{structural conditions}, namely the rank condition \eqref{eq:rank-condition-1} and its dual version, which we view through the lens of group theory rather than linear algebra.
The previous work of \cite{BCCT08}, \cite{BCCT10}\, \cite{CDKSY} and \cite{BJ} treats all the elementary groups; except in the case of vector spaces, we need to use approximation theory to extend the arguments of these authors to deal with general examples of the basic types of group.

Here is our plan of the rest of the paper. 
Section \ref{sec:background} introduces our notation and presents the necessary background material, such as Weil's formula for integration and some structure theory.
In Section \ref{sec:sim} we study how the Brascamp--Lieb constant behaves when we pass to compact quotient groups and open subgroups, clarify the passage to nondegenerate and canonical data, and show how the Brascamp--Lieb and dual Brascamp--Lieb constants behave when $G$ is approximated by elementary groups. 
This allows us to consider and slightly extend the known results for particular types of LCA groups  in Section \ref{sec:rank-cond}.
Then in Section \ref{sec:gen} we consider the general case of the inequality and establish Theorem \ref{thm:intro-thm-2}. 
In Section \ref{sec:rem} we collect some further remarks that may be helpful in the computation of Brascamp--Lieb constants in the general setting of LCA groups.
\begin{rem}
Fourier duality is one of several notions of duality that may be formulated in the general setting of this paper. For example, the Brascamp--Lieb constant $\BL(G, \bss, \bsp)$ may be characterised in terms of a form of entropy subadditivity via Legendre duality  \cite{CCE}. We refer to \cite{CHV20,BT} and the references there for further examples and context.
\end{rem}
\section{Notation and Background}\label{sec:background}

Here we recall some basic results on the dual group, the Haar measure, on homomorphisms, and on the structure of LCA groups.
Most of these results are contained in \cite{HR63, HR70}; 
other texts on LCA groups, such as \cite{Reiter-Stegeman, Rudin}, may also be useful.

\subsection{Dual groups}
The LCA group $G$ has a dual group $\hat G$, which is the group $\Hom(G,\T)$ of all homomorphisms $\xi: G \to \T$, where $\T$ denotes the multiplicative group of complex numbers of modulus one.
The dual group of $\hat G$ is $G$; the dual of $\R^a$ is $\R^a$; the dual of a compact group is discrete, and vice versa.

If $H$ is a closed subgroup of $G$, then $H$ with the relative topology and $G/H$ with the quotient topology are both LCA groups; $G/H$ is discrete if and only if $H$ is open.
Given $x \in G$, we write $\dot x$ for the coset $xH$ in $G/H$; a function $f$ on $G/H$ may be identified with an $H$-invariant function $\dot f$ on $G$ by setting $\dot{f}(x) = f(\dot x)$.

If $H$ is a closed subgroup of $G$, then the annihilator $\Ann(H, \hat G)$ of $H$ in $\hat G$, usually written $H^\perp$, is the set of all $\xi \in \hat G$ such that $\xi(y) = 1$ for all $y \in H$.
The annihilator of $H^\perp$ in $G$ is $H$.
Then $H^\perp$ is the dual group of $G/H$ and $H$ is the dual group of $\hat G/H^\perp$.
In particular, if $H$ is a closed subgroup of $G$, then the following are equivalent: $H$ is open; $G/H$ is discrete; and $H^\perp$ is compact.

\subsection{Haar measure}

Suppose that $G$ is a not necessarily abelian locally compact group.  Then $G$ has a unique (up to positive multiples) Haar measure, a left-translation-invariant regular Borel measure.
Haar measures are often normalised so that compact groups have total measure $1$ and discrete groups have counting measure (but this requirements are incompatible for finite groups, which are both discrete and compact).
We write $\wrt{x}$, $\wrt{y}$, and so on, in integrals over groups to indicate integration with respect to the Haar measure of the appropriate group, and $\abs{E}$ to indicate the Haar measure of a set $E$.
An open subgroup $H$ of $G$ is automatically closed, and we usually take the Haar measure on $H$ to be the restriction of that on $G$.

We now consider abelian groups only.
Take a Haar measure on the closed subgroup $H$ of $G$; then for any $f$ in the space $S(G)$  of test functions (to be clarified shortly) on $G$, the function
\[
x \mapsto \int_H f(x y) \wrt{y}
\]
is an $H$-invariant continuous function on $G$, which may be identified with a function in $S(G/H)$.
Further, the Haar measure on $G/H$ may be normalised in such a way that
\begin{equation}\label{eq:Weils-formula}
\int_G f(x) \wrt{x} = \int_{G/H} \int_H f(xy) \, \wrt{y} \wrt{\dot x}.
\end{equation}
This expression, sometimes called Weil's formula (see \cite[p.~88]{Reiter-Stegeman}), extends to more general measurable functions provided that appropriate precautions are observed.
We note that it is possible (see, e.g., \cite[p.~258]{Gaal}) to find a continuous nonnegative-real-valued function $g$ on $G$ such that
\[
\int_{H} g(xy) \wrt{y} = 1 
\qquad\forall x \in G.
\]

A particular case of Weil's formula arises when the subgroup $H$ of $G$ is compact: if we normalise $H$ to have measure $1$ and $G/H$ so that Weil's formula holds, then we say that $H$ and $G/H$ have the \emph{standard compact and compact quotient normalisations}.
In this case, if $f$ on $G$ is constant on cosets of $H$, then $f$ gives rise naturally to a function $\dot f$ on $G/H$, and
\[
\int_{G} f(x) \wrt{x} = \int_{G/H} \dot{f}(\dot{x}) \wrt{\dot{x}}.
\]

Another special case of Weil's formula is when $H$ is an open subgroup of $G$.
In this case, it is natural to take the Haar measure on $H$ to be the restriction of that of $G$, and the Haar measure on the discrete group $G/H$ to be counting measure; these are the \emph{standard open and open quotient normalisations}.

If $H$ is either open and not compact, or compact and not open, we choose the Haar measures on $H$ and $G/H$ as just described, unless otherwise stated.
In case $H$ is both compact and open, the choice of Haar measures on $H$ and $G/H$ needs to be specified.

Convolution on $G$ is defined in the standard way: for suitable functions $f$ and $g$,
\[
f * g (x) = \int_{G} f(y) g(y^{-1}x) \wrt y
\]

\subsection{Fourier transformation}
We define the Fourier transform $\hat f: \hat{G} \to \C$ of a suitable function $f$ on $G$ by
\[
\hat f(\xi) = \int_{G} \bar{\xi}(x) \fn f(x) \wrt x
\qquad\forall \xi \in \hat{G}.
\]
We sometimes write $\Fourier f$ instead of $\hat f$.
The Haar measure on the LCA group $\hat{G}$ may be normalised such that the Plancherel and inversion formulas hold, for ``suitable'' functions on $G$, by which we mean integrable functions on $G$ whose Fourier transforms $\hat f$ are also integrable on $\hat G$. 

\subsection{Weil's generalised Poisson summation formula}
Let $G$ be a closed subgroup of a product group $\bsG$ and $G^\perp$ be the annihilator of $G$ in $\bsGh$.
If the Haar measures of $\bsG$, $G$ and $\bsG/G$ are normalised such that Weil's integration formula \eqref{eq:Weils-formula} holds, and the Haar measures of $\bsGh$, $G^\perp$ and $\bsGh/G^\perp$ are normalised such that the Fourier inversion theorem holds, then Weil's integration formula \eqref{eq:Weils-formula} also holds for $\bsGh$, $G^\perp$ and $\bsGh/G^\perp$, and further,
\[
\int_{G} (f_1 \otimes \dots \otimes f_J)(x) \,dx
= \int_{G^\perp} (\hat f_1 \otimes \dots \otimes \hat f_J)(y) \,dy.
\]

\subsection{Test functions}
There are various candidates for the space of test functions $S(G)$ on $G$.
One possibility is to consider $G$ as $V \times H$, where $V$ is a real vector space and $H$ has a compact open subgroup, and to take the space of finite sums of product functions $f_1 \otimes f_2$, where $f_1$ is a Schwartz function on $V$ and $f_2$ has compact support on $H$ while $\hat f_2$ has compact support on $\hat G$.
A remediable source of discomfort with this choice is that the decomposition of $G$ as $V \times H$ may not be unique.
An alternative choice for the test function space that is independent of any structural assumptions is Feichtinger's minimal Segal algebra \cite{Feich, ja18}

The important properties that we need, which are shared by both these spaces, are that $\Fourier (S(G)) = S(\hat G)$, and that $S(G)$ is a dense subspace of each $L^p(G)$ space, in norm when $1 \leq p < \infty$ and in the weak-star topology when $p = \infty$.
Further, let $S(G)^+$ denote the subset of $S(G)$ of nonnegative-real-valued test functions; then there are nontrivial functions in $S(G)^+$ with support contained in arbitrarily small neighbourhoods of the identity in $G$, and there are functions in $S(G)^+$ taking the value $1$ on arbitrarily large compact sets in $G$.

\subsection{Approximation by ``smaller'' groups}
Every LCA group $G$ may be written in the form $V \times H$, where $V$ is a vector subgroup and $H$ has a compact open subgroup, $K$ say.
Let $(H_\alpha)_{\alpha \in \Alpha}$ be the net of compactly generated open subgroups of $H$ that contain $K$, ordered by inclusion, and let $G_\alpha = V \times H_\alpha$.
Then $G_\alpha \subseteq G_{\alpha'}$ when $\alpha \leq \alpha'$ and $G = \cup_{\alpha \in \Alpha} G_\alpha$.

The dual group $\hat G$ of $G$ may be identified with $\hat V \times \hat H$, and the annihilator $K^\perp$ of $K$ in $\hat H$ is open in $\hat H$.
Let $(\hat{H}_{\beta})_{ \beta\in \Beta}$ be the net of compactly generated subgroups of $\hat H$ that contain $K^\perp$, ordered by inclusion, and write $K_\beta$ for $\Ann(\hat V \times \hat H_\beta, G)$.
Then $K_\beta \subseteq K$, and as $\hat H_\beta$ becomes larger, $K_\beta$ becomes smaller; moreover $\cap_{\beta \in \Beta} K_\beta = \{ e\}$.
We may view $K_\beta$ as a subgroup of $H$ or of $H_\alpha$ for any $\alpha$, and hence also of $G$ or of $G_\alpha$.

Functions on $G$ with support in $G_\alpha$ may be identified with functions on $G_\alpha$.
Similarly, functions on $G$ whose Fourier transforms have support in $\hat G_\beta$ may be identified with functions on $G$ that are constant on cosets of $K_\beta$, or with functions on $G/K_\beta$.
We write $G_{\alpha,\beta}$ for $G_\alpha / K_\beta$; then both this group and its dual group are compactly generated.

The natural Haar measure on $G$ is the product of the Lebesgue measure on $V$ and the Haar measure of $H$, normalised such that $K$ has total mass $1$.
This induces the standard open and compact quotient Haar measure on each $G_{\alpha,\beta}$.

The group $G_{\alpha,\beta}$ is compactly generated with a compactly generated dual, and is of the form $\R^a \times \T^b \times \Z^c \times F$, where $a, b, c \in \N$ and $F$ is finite.
While $a$ is fixed, the indices $b$ and $c$ and the finite group $F$ may depend on $\alpha$ and $\beta$.
It should be noted that the standard Haar measure of $\T^b$ is $1$, and the standard Haar measure on $\Z^c$ is counting measure. 
However $F$ is automatically equipped with a measure that may be neither counting measure nor a probability measure: the image of $K$ in $F$ has total measure $1$.

Let $\mu_{\beta}$ denote the normalised Haar measure on $K_\beta$, viewed as a compactly supported measure on $G$.
Then
\[
\Fourier(\mu_\beta * f) = \chi_{\hat G_\beta} \Fourier(f)
\]
for all $f \in L^1(G)$, where $\chi_{\hat G_\beta}$ denotes the characteristic function of $\hat G_\beta$.

We may define an operator $\calE_{\alpha,\beta}$ on $L^p(G)$, by $\calE_{\alpha,\beta} f = \mu_\beta * (\chi_{G_\alpha} f)$.
Then
\[
\calE_{\alpha,\beta} f
= \chi_{G_\alpha}(\mu_\beta *  f)
= \Fourier^{-1} (\chi_{\hat G_\beta} \Fourier (\chi_{G_\alpha} f))
= \chi_{G_\alpha} \Fourier^{-1} (\chi_{\hat G_\beta} \Fourier ( f)) .
\]

\begin{lem}\label{lem:cond-exp}
For all $f \in S(G)$,
\[
\norm{\calE_{\alpha,\beta}f}_{L^p(G)} \leq \norm{f}_{L^p(G)}
\qquad\text{and}\qquad
f = \lim_{\alpha,\beta} \calE_{\alpha,\beta}f
\quad\text{in $L^p(G)$},
\]
and
\[
\norm{\Fourier \calE_{\alpha,\beta}f}_{L^p(\hat G)}
\leq \norm{\Fourier f}_{L^p(\hat G)}
\qquad\text{and}\qquad
\Fourier f = \lim_{\alpha,\beta} \Fourier (\calE_{\alpha,\beta}f)
\quad\text{in $L^p(\hat G)$};
\]
convergence is in norm when $1 \leq p < \infty$ and in the weak-star topology when $p = \infty$, and
$\norm{\calE_{\alpha,\beta}f}_{L^p(G)}$ and $\norm{\Fourier \calE_{\alpha,\beta}f}_{L^p(\hat G)}$ increase, and converge to $\norm{f}_{L^p(G)}$ and $\norm{\Fourier f}_{L^p(\hat G)}$.
\end{lem}

\begin{proof}
The operator $\calE_{\alpha,\beta}$ is contractive on $L^p(G)$ for $p\in [1,\infty]$ as multiplication by characteristic functions and convolution with probability measures are both so.
Further, $\calE_{\alpha,\beta}f \to f$ in $L^p(G)$ for all $p \in [1,\infty]$ for all $f$ in $S(G)$.
The first part now follows by density of $S(G)$ in $L^p(G)$.

Turning to the second part, given that $\Fourier (\calE_{\alpha,\beta}f) = (\hat \calE_{\alpha,\beta}\Fourier f)$, where
\[
\hat \calE_{\alpha,\beta}g = \chi_{\hat G_\beta}(\mu_\alpha *  g)
= \mu_\alpha * (\chi_{\hat G_\beta} g),
\]
and $\mu_\alpha$ is the normalised Haar measure on the compact group $(G_\alpha)^\perp$, the dual result follows from a similar argument.

As $G_\alpha$ increases to $G$, $\norm{ \chi_{G_\alpha} f}_{L^p(G)}$ increases to $\norm{ f}_{L^p(G)}$.
Likewise $\norm{ \mu_\beta * f}_{L^p(G)}$ increases to $\norm{ f}_{L^p(G)}$ as $K_\beta$ shrinks to $\{e\}$.
\end{proof}

\subsection{More structure theory}
The classical structure theory of LCA groups reduces matters to elementary groups.
Unfortunately however, homomorphisms of elementary groups do not always respect the decomposition $\R^a \times \T^b \times \Z^c \times F$, even if they are open and surjective, and our analysis relies on having a structure that is respected by homomorphisms.
Our point of view is not standard, but all the underlying ideas are well known.

The connected component $G_{c}$ of $G$ (by which we mean the connected component of the identity) is its unique maximal closed connected subgroup $G_{c}$, and $G_c = \bigcap_{G_1 \in \calO(G)} G_1$, where $\calO(G)$ is the collection of open subgroups of $G$ (see \cite[(7.8)]{HR63}).
The quotient group $G/G_{c}$ is totally disconnected, and its connected component is trivial.
By a theorem of van Dantzig \cite{vD36}, the compact open subgroups of $G/G_c$ form a neighbourhood base for the topology of $G/G_c$ at the identity.

By \cite[(9.2)]{HR63}, if $x$ is in a LCA group $G$, then either $\{x^n : n \in \Z\}$ is an infinite discrete subgroup of $G$ or $\{x^n : n \in \Z\}$ is relatively compact; we say that $x$ is bounded in the second case.
The set $G_{b}$ of all bounded elements of $G$ is a closed subgroup of $G$; see \cite[(9.26)]{HR63} (this reference uses the terminology ``compact elements'').

\begin{lem}
Let $G$ be an LCA group.
Then 
\begin{enumerate}
  \item the quotient group $G/G_{b}$ has no nontrivial bounded elements;
  \item $G_{c} \cap G_{b}$ is the (unique) maximal connected compact subgroup of $G$, and is equal to $(G_{c})_{b}$ and $(G_{b})_{c}$;
  \item $G_{c} G_{b}$ is an open subgroup of $G$, and $x \in G_{c} G_{b}$ if and only if $x G_{b} \in (G/G_{b})_{c}$ if and only if $x G_{c} \in (G/G_{c})_{b}$.
\end{enumerate}
\end{lem}

\begin{proof}
Given any $x \in G$, there is a compactly generated open subgroup $G_0$ of $G$ that contains $x$, which is of the form $\R^a \times \Z^b \times K$, where $K$ is compact; clearly, 
\[
G_0 \cap G_{b} = (G_0)_{b} \simeq K
\quad\text{and}\quad
G_0 \cap G_{c} = G_{c} \simeq \R^a \times K_{c}. 
\]
Now, for instance, if $xG_{b}$ is bounded in $G/G_{b}$, then $x \in K \subseteq G_{b}$. 
The other assertions are similarly straightforward.
\end{proof}

Recall that $G^{1} := G_{b} \cap G_{c}$; $G^{2} := G_{c}/(G_{b} \cap G_{c})$; $G^{3} := (G_{b} G_{c}) /G_{c}$; and  $G^{4} := G/(G_{b} G_{c})$.
From \cite[(24.17)]{HR63}, $\Ann(G_{b}, \hat{G}) = (\hat G)_{c}$, and $\Ann(G_{c},\hat G) = (\hat G)_{b}$.
Hence $\Ann(G_{b}\cap G_{c}, \hat{G}) = (\hat G)_{b}(\hat G)_{c}$, and $\Ann(G_{b} G_{c},\hat G) = (\hat G)_{b} \cap (\hat G)_{c}$.
It follows that the compact connected group $G^{1}$ may be identified with the dual of the discrete torsion-free group $(\hat{G})^4$, and the discrete torsion-free group $G^{4}$ may be identified with the dual of the connected compact group $(\hat{G})^1$.
Further, $(G_{b} G_{c}) / (G_{b} \cap G_{c})$ is a direct product of a vector group (its connected component) and the totally disconnected closed subgroup of all bounded elements, isomorphic to $G_{c} / (G_{b} \cap G_{c})$ and $(G_{b} G_{c}) / G_{c}$ respectively, and the same holds for $(\hat{G})_{b} / (\hat{G})_{b} \cap (\hat{G})_{c} $; the two vector groups are dual, as are the two totally disconnected groups.
It should now be clear that if $G$ is of the form $\R^a \times \T^b \times \Z^c \times F$, where $a, b, c \in \N$ and $F$ is a finite group, then $G^1 \simeq \T^b$, $G^2 \simeq \R^a$, $G^3 \simeq F$ and $G^4 \simeq \Z^c$.

It is natural to normalise the Haar measure on $G^1$ to have total mass $1$ and to give $G/G^1$ the standard compact quotient measure.
Likewise, we normalise the Haar measures on $G_bG_c$ and on $G^4$ to be the standard open subgroup and open quotient measures.
The Haar measure on $G$ then induces a Haar measure on $G_bG_c/(G_b \cap G_c)$, which is the direct product of a vector group and a totally disconnected group of bounded elements.
Fixing the Haar measure on one of the factors then determines that on the other factor.
The choice here can be arbitrary (at least if both $G^2$ and $G^3$ are nontrivial), as any change multiplies $\BL(G^2,\bss^2,\bsp)$ by a factor and divides $\BL(G^3,\bss^3,\bsp)$ by the same factor, leaving their product unchanged.

It may be worth providing some examples of nonelementary LCA groups, and how the approximation process works in these cases.

\begin{example}
The rational numbers $\Q$, equipped with the discrete topology and counting measure, is a countable torsion-free group that is not finitely generated.
Indeed, if we take a finite set $S$ of rational numbers, and write them all with a lowest common denominator $N$, then it is clear that all these elements lie inside the additive group $\Z /N$ of all rational numbers with denominator $N$, and so $\lip S \rip$, the group generated by $S$, is a proper subgroup of $\Q$. 
In fact, $\lip S \rip = \Z/N$.
The group $\Q$ is the inductive limit $\bigcup_S \lip S \rip$ as $S$ increases and the common denominator $N$ increases.
For this group $G$, the groups $G^1$, $G^2$ and $G^3$ are all trivial, while $G = G^4$.
The approximating finitely generated groups are all copies of the group of integers $\Z$.
It is natural to equip $\Q$ and $G^4$ with counting measure. 
\end{example}

\begin{example}
The Bohr compactification $\bar{\R}$ of the real numbers is a nonseparable connected compact group, whose dual group is $\R_d$, that is, $\R$ with the discrete topology.
A finitely generated subgroup $\Gamma$ of $\R_d$ is isomorphic to $\Z^m$ for some positive integer $m$, the annihilator $\Ann(\Gamma, \bar{\R})$ is a closed subgroup of $\bar{\R}$, and the quotient group is isomorphic to $\T^m$.
For this group $G$, the groups $G^2$, $G^3$ and $G^4$ are all trivial, while $G = G^1$.
The approximating groups with finitely generated duals are all of the form $\T^m$ for some finite (but unbounded) $m$.
It is natural to equip both $\bar{\R}$ and the approximating groups with measures of total mass $1$.
\end{example}

\begin{example}
Let $p$ be a prime number.
The field of $p$-adic numbers $\Q_p$ is the completion of $\Q$ in the translation-invariant metric $d$ such that $d(x,0) = p^{-n}$, when $x = r p^n/s$, where $r,s \in \Z$ are coprime to $p$.
The ring of $p$-adic integers $\O_p$ is the completion of $\Z$ in $\Q_p$; $\O_p$ is a compact open subgroup of $\Q_p$.
Equipped with its additive structure, $\Q_p$ is an example of a totally disconnected LCA group. 
It has an increasing family of compact open subgroups $K_n$, where $n \in \Z$, with the property that $\bigcap_n K_n = \{e\}$, $\bigcup_n K_n = \Q_p$, and $K_{n+1}/K_n$ is isomorphic to $\Z_p$, the ring of integers modulo $p$; we may arrange matters so that $K_0 = \O_p$.
For this group $G$, the groups $G^1$, $G^2$ and $G^4$ are all trivial, while $G = G^3$.
The approximating groups that are compactly generated with compactly generated dual groups are all finite $p$-groups, that is, all elements are of order $p^n$ for some positive integer $n$.
It is natural to normalise the Haar measure on $\Q_p$ so that its restriction to $\O_p$ is a probability measure and equip the approximating groups with the standard open and compact quotient normalisations.
\end{example}

\begin{example}
The ring of adeles is used in algebraic number theory. 
It consists of the restricted direct product $\R \times \Q_2 \times \Q_3 \times \dots$, where restricted direct product means that each element $(x_0, x_2, x_3, \dots)$ has the property that all but finitely many of the $x_p$ lie in $\O_p$.
For this group $G$, the groups $G^1$ and $G^4$ are both trivial, while $G^2 = \R$ and $G^3$ is the restricted direct product of the various groups of $p$-adic numbers.
We may view $\{0\} \times \O_2 \times \O_3 \times \dots$ as the ``natural choice'' of totally disconnected compact subgroup of $G$.
The approximating groups are products of $\R$ and finitely many finite $p$ groups, obtained by factoring out $\{0\} \times K_2 \times K_3 \times \dots$, where all but finitely many of the compact open subgroups $K_p$ of $\Q_p$ are equal to $\O_p$. 
The natural Haar measures are the product of the Haar measures on $\R$ and on the factors $\Q_p$ on the adeles, and the standard open and compact quotient measures on the approximating groups.
\end{example}

\subsection{Open and proper homomorphisms}\label{Subsect:proper}
The results from Lemma \ref{lem:good-or-bad} to Proposition \ref{prop:new-homos-from-old} apply to all locally compact groups, including those that are not abelian.

\begin{lem}\label{lem:good-or-bad}
Suppose that $\sigma: G \to H$ is a homomorphism of not necessarily abelian locally compact groups.
If $\sigma$ is open, then $|\sigma(U)| > 0$ for all nonempty open subsets $U$ of $G$.
If $\sigma$ is not open, then $|\sigma(V)| = 0$ for all compact subsets $V$ of $G$.
Finally, $\sigma$ is open if and only if $\sigma(G_0)$ is open in $H$ for all open subgroups $G_0$ of $G$.
\end{lem}

\begin{proof}
If $U$ and $V$ are nonempty relatively compact open sets in $G$, then $V$ may be covered by finitely many translates of $U$ and vice versa, and $\sigma(V)$ may be covered by finitely many translates of $\sigma(U)$  and vice versa.
It follows that $\sigma(U)$ has null measure if and only if $\sigma(V)$ has null measure.

In light of this observation and the fact that the Haar measure of a nonempty open set is strictly positive, it suffices to prove that if $U$ is a nonempty relatively compact open set in $G$ and $\sigma(U)$ has positive Haar measure in $H$, then $\sigma(U)$ is open in $H$.

Given any $z \in U$, $z^{-1}U$ is a relatively compact open neighbourhood of the identity in $G$, and by \cite[p.~18]{HR63}, it is possible to find a relatively compact open neighbourhood $V$ of the identity in $G$ such that $V^2 \subseteq z^{-1}U$.
It will suffice to show that $\sigma(V^2)$ contains an open neighbourhood $W$ of the identity in $H$, for then $\sigma(z)W$ is an open neighbourhood of $\sigma(z)$ in $H$ that is contained in $\sigma(zV^2)$, and hence in $\sigma(U)$.

Since $V$ is relatively compact and $\sigma$ is continuous, $\sigma(V)$ is relatively compact in $H$ and has finite Haar measure, which is not $0$ since the measure of $\sigma(U)$ is not $0$.
Then the convolution on $H$ of the characteristic function of $\sigma(V)$ with itself, given by the formula
\[
\chi_{\sigma(V)} \ast \chi_{\sigma(V)}(x)
= \int_H \chi_{\sigma(V)}(y) \, \chi_{\sigma(V)}(y^{-1}x) \wrt{y},
\]
is a continuous function on $H$ (see \cite[p.~4]{Rudin}) that does not vanish at the identity but vanishes outside $\sigma(V^2)$, and $\{y \in H: \chi_{\sigma(V)} \ast \chi_{\sigma(V)}(y) \neq 0\}$ is the desired set $W$.

If $\sigma$ is open, then $\sigma(G_0)$ is open in $H$ for all open subgroups $G_0$ of $G$.
Conversely, if this last condition holds and $U$ is a relatively compact open subset of $G$, then $V = U \cup U^{-1}$ is open, so $\cup_{n \in \Z} V^n$ is an open subgroup of $G$.
Since $\sigma(\cup_{n \in \Z} V^n)$ is open, it has positive measure, and so at least one set $\sigma(V^n)$ has positive measure. 
Since $V^n$ is a relatively compact open set in $G$ and $\sigma(V^n)$ has positive measure, $\sigma$ is open.
\end{proof}

A mapping is said to be \emph{proper} if the inverse image of every compact set is compact.

\begin{lem}\label{lem:proper-homos}  
Suppose that $\sigma: G \to H$ is a homomorphism of not necessarily abelian locally compact groups.
Then $\sigma$ is proper if and only if $\ker(\sigma)$ is compact, $\sigma(G)$ is closed in $H$, and $\sigma: G \to \sigma(G)$ is open when $\sigma(G)$ carries the relative topology as a subspace of $H$.
\end{lem}

\begin{proof}
First, if $\ker(\sigma)$ is not compact, then $\sigma^{-1}(\{e\})$ is not compact and $\sigma$ is not proper.
If $\ker(\sigma)$ is compact, and $\dot\sigma$ is the induced homomorphism from $G/\ker(\sigma)$ to $H$, then $\sigma$ is proper if and only if $\dot\sigma$ is proper.
Hence we may assume that $\sigma$ is injective.

If $x_\nu \in G$ and $\sigma(x_\nu) \to y$ in $H$, then $\sigma(x_\nu)$ is eventually in any compact neighbourhood $V$ of $y$ in $H$, whence $x_\nu$ lies eventually in the compact set $\sigma^{-1}(V)$ in $G$.
If $x$ is a limit point of $(x_\nu)_\nu$, then there is a subnet $(x_\mu)_\mu$ of $(x_\nu)_\nu$ that converges to $x$; whence $(\sigma(x_\mu))_\mu$ converges to $\sigma(x)$.
It follows that $\sigma(x) = y$, which shows that $y \in \sigma(G)$ and so $\sigma(G)$ is closed in $H$.

Next, define $\tilde{H} := \sigma(G)$ and $\tilde\sigma: G \to \tilde{H}$ by $\tilde\sigma(x) := \sigma(x)$ for all $x\in G$.
Then $\sigma$ is proper if and only if $\tilde{\sigma}$ is proper.
Hence we may assume that $\sigma$ is surjective.

Assume for the rest of this proof that $\sigma$ is a (continuous) algebraic isomorphism from $G$ onto $H$. 
We shall show that $\sigma$ is proper if and only if it is a homeomorphism.

Suppose that $\sigma$ is proper.
Take a compact subset $K$ of $H$ of positive Haar measure, then $\sigma^{-1}(K)$ is compact in $G$.
Let $V$ be a relatively compact open set in $G$.  
Then finitely many translates $x_1V$, \dots, $x_mV$ of $V$ cover $\sigma^{-1}(K)$, and so finitely many translates $\sigma(x_1)\sigma(V)$, \dots, $\sigma(x_m)\sigma(V)$ of $\sigma(V)$ cover $K$.
It follows that $\sigma(V)$ has positive Haar measure in $H$, and hence $\sigma$ is open by Lemma \ref{lem:good-or-bad}.

Suppose now that $\sigma$ is open. 
Take a compact subset $K$ of $H$ and a cover of $\sigma^{-1}(K)$ by open sets $U_\gamma$.
Then $(\sigma^{-1}(U_\gamma))$ is a cover of $K$ by open sets, so has a finite subcover; this gives us a finite subcover of $(U_\gamma)$, whence $\sigma^{-1}(K)$ is compact.
\end{proof}

We are interested in vector homomorphisms $\bss: G \to \bsG$, where $\bsG := G_1 \times\dots\times G_J$.

\begin{lem}\label{lem:proper-vector-homos}
Suppose that $\bss: G \to \bsG$ is a proper vector homomorphism.
Then $\ker(\bss)$ is compact, the subgroup $\bss(G)$ is closed in $\bsG$ and each $\sigma_j$ is open. 
\end{lem}

\begin{proof}
By Lemma \ref{lem:proper-homos}, $\ker(\bss)$ is compact, $\bss(G)$ is closed in $\bsG$, and $\bss$ is open as a mapping from $G$ to $\bss(G)$ with the relative topology as a subspace of $\bsG$.
Each $\sigma_j$ is open since $\sigma_j = \pi_j \circ \bss$ and the projection $\pi_j$ is open.
\end{proof}

The mapping $\bss: (x,y) \mapsto (x+y, x-y)$ is a continuous homomorphism of $\R \times \R_d$ onto $\R^2$, and $\ker(\bss)$ is compact, the subgroup $\bss(G)$ is closed in $\bsG$  and each $\sigma_j$ is open, but $\bss$ is not proper.

\begin{prop}\label{prop:new-homos-from-old}
Suppose that $\sigma:G \to H$ is a locally compact group homomorphism and $N$ is a closed normal subgroup of $G$ such that $\sigma(N)$ is normal in $H$. 
Then the restricted mapping $\rho: N \to \sigma(N)\afterbar$ and the induced quotient mapping $\dot\sigma: G/N \to H/(\sigma(N)\afterbar)$ are also locally compact group homo\-mor\-phisms.
Moreover, $\sigma$ is proper if and only $\sigma(N)$ is closed in $H$ and $\rho$ and $\dot{\sigma}$ are proper.
\end{prop}

\begin{proof}
It is straightforward to check that $\rho$ and $\dot{\sigma}$ are homomorphisms of locally compact groups.

Suppose that $\sigma$ is proper; then $\sigma(N)$ is closed in $H$, so $\sigma(N)\afterbar = \sigma(N)$.
If $V$ is a compact subset of $\sigma(N)$, it is also a compact subset of $H$ and $\sigma^{-1}(V)$ is compact; thus $\rho$ is proper.
If $V$ is a compact subset of $H/\sigma(N)$, there exists a compact subset $W$ of $H$ such that $W\sigma(N) = V$. 
Clearly, $\dot{\sigma}^{-1}(V) = \sigma^{-1}(W \sigma(N)) = \sigma^{-1}(W) N$; thus $\dot{\sigma}$ is proper.

Conversely, suppose that $\rho$ and $\dot{\sigma}$ are proper.
Then $\sigma(N)$ is open in $\sigma(N)\afterbar$, and hence closed in $\sigma(N)\afterbar$, that is $\sigma(N)$ is closed.
Further, $K := \ker(\sigma)$ is compact. 
The homomorphisms $\sigma$, $\rho$ and $\dot{\sigma}$ induce homomorphisms $\sigma'$, $\rho'$ and $\dot{\sigma}'$ 
from $G / K$ to $H$, from $N K/K$ to $\sigma(N)$, and from $G / NK$ (which is isomorphic to $(G/K) / (NK/K)$) and each of the induced homomorphisms is proper if and only if the inducing homomorphism is proper, so we may and shall assume without loss of generality that $\sigma$ is a bijection.

To show that $\sigma$ is proper, we take a subset $W$ of $G$ such that $\sigma(W)$ is compact, which implies that $W$ is closed in $G$, and a net $(x_\gamma)$ in $W$; we must show that the net $(x_\gamma)$ has a convergent subnet.
Since $\dot{\sigma}$ is proper, $WN$ is compact in $G/N$, and by passing to a subnet if necessary, we may suppose that the net $(x_\gamma N)$ converges in $G/N$, to $zN$ say, where $z \in G$.
Take a compact neighbourhood $V$ of $z$ in $G$ such that $VN$ is a neighbourhood of $zN$ in $G/N$; again by passing to a subnet if necessary, we may suppose that $x_\gamma N \in VN$ for all $\gamma$. 
Take $z_\gamma$ in $V$ such that $z_\gamma N = x_\gamma N$; by yet again passing to a subnet if necessary, we may suppose that $z_\gamma \to z$ in $V$.
Now $(z_\gamma^{-1} x_\gamma) N = (z_\gamma N)^{-1} x_\gamma N = eN$ in $G/N$ and hence $\sigma (z_\gamma^{-1} x_n)$ lies in the compact subset $\sigma(V^{-1}) \sigma(W) \cap \sigma(N)$ of $\sigma(N)$. 
Now $\rho = \sigma|_{N}$ is proper, each $z_\gamma^{-1} x_\gamma$ lies in a compact subset of $N$, so by passing to a subnet if necessary for the last time, we may assume that $(z_\gamma^{-1} x_\gamma)$ converges and hence $(x_\gamma)$ converges.
\end{proof}

\begin{rem}
In the case where $G$ is abelian, normality is not an issue.
In the general case, we could remove the assumption of normality, and deal with a continuous quotient mapping $\dot{\sigma}$ of locally compact homogeneous spaces instead of a homomorphism.
\end{rem}

We now revert to the standing assumption that our groups are abelian.

\begin{cor}\label{cor:connected-to-connected}
If $\sigma: G \to H$ is a proper homomorphism of LCA groups, then $\sigma$ induces proper homomorphisms from $G_c$ to $H_c$ by restriction and from $G/G_c$ to $H/H_c$ by factoring, and proper homomorphisms from $G_b$ to $H_b$ by restriction and from $G/G_b$ to $H/H_b$ by factoring.
\end{cor}

\begin{proof}
Since $\sigma$ is open and $G_{c}$ is the intersection of all open subgroups of $G$, $\sigma(G_c)$ is the intersection of all open subgroups of $G_j$, that is, $(G_j)_{c}$.
Further, $x$ is a bounded element of $G$ if and only if $\sigma(x)$ is a bounded element of $H$, so $\sigma(G_b) = (\sigma(G))_b$.
We now apply Proposition \ref{prop:new-homos-from-old}.
\end{proof}

\section{Simplifications}\label{sec:sim}
In this section, we consider how the topology of the vector homomorphism $\bss$ is reflected in the Brascamp--Lieb constant, and how the constants on a ``big group'' are related to the Brascamp--Lieb constants on ``smaller groups'', such as subgroups or quotients. This allows us to reduce matters to nondegenerate data, clarify the sense in which such data are equivalent to canonical data, and reduce the computation of constants to the simpler setting of elementary LCA groups.

\subsection{Properness of $\bss$ and nondegenerate data}
Recall that $\bss$ is a ``vector homomorphism'' $(\sigma_1, \dots, \sigma_J)$ from $G$ to $\bsG := G_1 \times \dots \times G_J$.
We write $\pi_j$ for the canonical projection from $\bsG$ to $G_j$.

\begin{lem}\label{lem:bss-must-be-proper}
Let $(G, \bss,\bsp)$ be a Brascamp--Lieb datum, and assume that $\bss: G \to \bsG$ is not proper. 
Then $\BL(G, \bss, \bsp)$ and $\BLhat(G, \bss, \bspp)$ are infinite. 
\end{lem}

\begin{proof}
Since $\bss$ is not proper, there exists a compact subset $K$ of $\bsG$ such that $\bss^{-1}(K)$ is not compact.
Let $V$ be a compact set in $G$ such that $|V| > 0$.
We claim that there exists  an infinite set $\{x_m: m \in\N\}$ in $\bss^{-1}(K)$ such that the compact sets $x_mV$ in $G$ are pairwise disjoint.
Indeed, otherwise there would be finitely many $x_m$ in $\bss^{-1}(K)$ such that $x V \cap (\cup_{m} x_m V)$ is nonempty for all $x \in \bss^{-1}(K)$, and then $\bss^{-1}(K)$ would be a closed subset of the compact set $(\cup_{m} x_m V) V^{-1}$, which is absurd.

Now take $f_j \in S(G_j)^+$ such that $f_j(y) \geq 1$ for all $y \in \pi_j(K \bss(V))$.  
Then $f_1 \otimes \dots \otimes f_J \geq 1$ on $K \bss(V)$.
Hence $\prod_j f_j(\sigma_j(x)) \geq 1$ for all $x \in \bss^{-1}(K) V \supseteq \cup_m x_m V$, and
\[
\int_{G} \prod_j f_j(\sigma_j(x)) \wrt{x} \geq \sum_{m} \int_{x_m V} 1 \wrt{x} = \infty.
\]
It follows that $\BL(G,\bss,\bsp)$ and $\BLhat(G,\bss,\bspp)$ are infinite.
\end{proof}

\begin{cor}\label{cor:kernel-compact}
Suppose that $(G, \bss,\bsp)$ is a Brascamp--Lieb datum.
If $\cap_j \ker(\sigma_j)$ is a noncompact subgroup of $G$, or if any $\sigma_j$ is not open, then $\BL(G, \bss, \bsp)$ and $\BLhat(G, \bss, \bspp)$ are both infinite.
\end{cor}

\begin{proof}
Each of the hypotheses implies that $\bss$ is not proper, by Lemmas \ref{lem:proper-homos} and \ref{lem:proper-vector-homos}.
\end{proof}

\subsection{Passing to subgroups and quotients}

Our next theorems are concerned with restricting to open subgroups and factoring out compact subgroups.

\begin{thm}\label{thm:BL-open-sbgp}
Suppose that $(G, \bss,\bsp)$ is a  Brascamp--Lieb datum, that $\bss$ is proper, that $G_0$ is an open subgroup of $G$, and that $H_j$ is an open subgroup of $G_j$ such that $\sigma_j(G_0) \subseteq H_j$ for each $j$.
The the induced maps ${\sigma}^\circ_j: G_0 \to H_j$, defined to be ${\sigma_j}|_{G_0}$ with codomain $H_j$, are homomorphisms and the vector homomorphism $\bss^\circ$ is proper. 
Further, when $G_0$ and $H_j$ are equipped with standard open Haar measures,
\[
\BL(G_0, \bss^\circ, \bsp)
\leq \BL(G, \bss, \bsp)
\qquad\text{and}\qquad
\BLhat(G_0, \bss^\circ, \bspp)
\leq  \BLhat(G, \bss, \bspp).
\]
The inequalities are equalities when $G_0 = G$.
\end{thm}

\begin{proof}
This is a straightforward consequence of applying the definition of $\BL(G, \bss, \bsp)$ to input functions supported on $G_0$, and we leave the reader to provide the details.
\end{proof}

\begin{thm}\label{thm:BL-cpct-qtnt}
Suppose that $(G, \bss,\bsp)$ is a Brascamp--Lieb datum, that $\bss$ is proper, and that $N$ is a compact subgroup of $G$. 
Then each $\sigma_j(N)$ is a compact subgroup of $G_j$ and the induced quotient mappings $\dot\sigma_j: G/N \to G_j/\sigma_j(N)$ are homomorphisms.
Further, $\bsds$ is proper and when the Haar measures on $G/N$ and $G_j/\sigma_j(N)$ are given the compact quotient normalisations,
\[
\BL(G/N, \dot{\bss}, \bsp) \leq \BL(G, \bss, \bsp)
\qquad\text{and}\qquad
\BLhat(G/N, \dot{\bss}, \bspp) \leq \BLhat(G, \bss, \bspp) .
\]
If $N \subseteq \ker(\sigma_j)$ for all $j$, then equality holds.
\end{thm}

\begin{proof}
Proposition \ref{prop:new-homos-from-old} shows that $\bsds$ is proper.

Each function $f_j \in S(G_j/\sigma_j(N))$ gives rise naturally to a function $\dot{f}_j \in S(G_j)$ by the formula $\dot{f}_j(x) = f_j(x\sigma_j(N))$, and
\[
\bignorm{\dot{f}_j}_{L^p(G_j)} = \bignorm{f_j}_{L^p(G_j/\sigma_j(N))}.
\]
All the functions $\dot{f}_j$ are constant on cosets of $N$ in $G$, so
\[
\int_{G/N} \prod_{j=1}^J f_j( \dot{\sigma}_j(x) ) \wrt{x}
= \int_{G} \prod_{j=1}^J \dot{f}_j( \sigma_j(x) ) \wrt{x} .
\]
Thus in working on $G/N$ and considering all the inputs for computing $\BL(G/N, \dot{\bss}, \bsp)$, we are effectively working on $G$ and considering only some of the inputs for computing $\BL(G, \bss, \bsp)$, and the claimed inequality for the Brascamp--Lieb constants follows.
If $N \subseteq \ker(\sigma_j)$ for all $j$, then we are considering all the inputs, and equality holds.

A similar argument holds for the dual constants.
\end{proof}

\subsection{Canonical data}
In the previous section we showed that there is no loss of generality in assuming that all $\sigma_j$ are surjective and $\ker(\bss) = \{e\}$, that is, that $(G, \bss,\bsp)$ is nondegenerate.
As noted in the introduction, when $(G, \bss, \bsp)$ is nondegenerate, the vector isomorphism $\bss$ embeds $G$ as a closed subgroup $\bss(G)$ of $\bsG$.
This embedding is a homeomorphism from $G$ to $\bss(G)$ with the relative topology.
Indeed, $\bss$ is continuous, and if $U$ is open in $G$, then
\[
\bss(U) = [\sigma_1(U) \times \dots \times \sigma_J(U)] \cap \bss(G),
\]
which is an open set in $\bss(G)$ in the relative topology.  
As a result, a nondegenerate datum is indeed equivalent (as clarified in the introduction) to a canonical datum.

If $(G, \bss,\bsp)$ is canonical, then we may define $G^\perp$ to be $\Ann(G, \bsGh)$, the annihilator of $G$ in $\bsGh:=\hat{G}_1 \times\dots\times \hat{G}_J$.
Weyl's generalisation of Poisson's formula \eqref{posum} then implies that
\[
\BLhat(G, \bsG,\bspp) = \BL(G^\perp, \bsGh,\bspp) .
\]
This means that the results enunciated in Theorem \ref{thm:intro-thm-2}
for $\BLhat(G, \bsG,\bspp)$ 
follow from the results enunciated for $\BL(G, \bsG,\bsp)$, for a general  canonical datum.

It is not clear how to write the homomorphism version of the Brascamp--Lieb inequality in a dual form.
This issue does not arise in the vector group case, because there are no nontrivial compact subgroups, and no nontrivial open subgroups.

\subsection{Approximation by elementary groups} 

We recall from Section \ref{sec:background} that every LCA group $G$ contains open subgroups $G_\alpha$ and compact subgroups $K_\beta$, and the quotients $G_{\alpha,\beta} = G_\alpha / K_\beta$ are elementary LCA groups that ``approximate $G$''.

Given a homomorphism $\sigma_j:G \to G_j$, we define the homomorphism $\sigma_{j,\alpha,\beta}: G_{\alpha,\beta} \to G_{j,\alpha,\beta}$, where $G_{j,\alpha,\beta} = \sigma_j(G) / \sigma_j(K_\beta)$, by
\[
\sigma_{j,\alpha,\beta}(x K_\beta) = \sigma_j(x) \sigma_j(K_{\beta});
\]
when the $\bss$ are proper, so are the $\bss_{\alpha,\beta}$, by Theorems \ref{thm:BL-cpct-qtnt} and \ref{thm:BL-open-sbgp}.

\begin{thm}\label{thm:BL-approx}
With the notation above,
\[
\BL(G, \bss, \bsp)
= \lim_{\alpha,\beta} \BL(G_{\alpha,\beta}, \bss_{\alpha,\beta}, \bsp)
\]
and
\[
\BLhat(G, \bss, \bspp)
= \lim_{\alpha,\beta} \BLhat(G_{\alpha,\beta}, \bss_{\alpha,\beta}, \bspp).
\]
\end{thm}

\begin{proof}
We treat the constant $\BL(G, \bss, \bsp)$; the other constant is similar.

On the one hand,
\[
\BL(G_{\alpha,\beta}, \bss_{\alpha,\beta}, \bsp)
\leq \BL(G, \bss, \bsp),
\]
because for the right hand term we maximise the integral
\[
\labs \int_{G} \prod_{j=1}^J f_j( \sigma_j(x) ) \wrt{x} \rabs
\]
over all $f_j \in L^p(H_j)$ of norm at most one, and for the left hand term we maximise the same expression over all such $f$ with additional support requirements on $f$ and $\Fourier f$.
Further, the constants $\BL(G_{\alpha,\beta}, \bss_{\alpha,\beta}, \bsp)$ form an increasing net.

On the other hand, suppose that $\BL(G_{\alpha,\beta}, \bss_{\alpha,\beta}, \bsp) \leq C$ for all $\alpha$ and $ \beta$.
The mappings $\calE_{\alpha,\beta}$ pass naturally to mappings $\calE_{\alpha,\beta,j}$ on the groups $G_{\alpha,\beta,j} = \sigma_j(G_{\alpha,\beta})$ and $f_j = \lim_{\alpha,\beta} \calE_{\alpha,\beta,j}f_j$, much as in Lemma \ref{lem:cond-exp}.
Hence we may write each $f_j$ as a sum $\sum_m f_{j,m}$, where each $f_{j,m}$ is equal to $\calE_{\alpha,\beta,j}f_{j,m}$, for suitable $\alpha$ and $\beta$, and $\sum_m \norm{f_{j,m}}_{L^{p_j}(G_j)} \eqsim \norm{f_{j}}_{L^{p_j}(G_j)}$.

Now
\begin{align*}
\labs \int_{G} \prod_{j=1}^J f_j ( \sigma_j(x) ) \wrt{x} \rabs
&\quad\leq \sum_{m_1, \dots, m_J} \labs\int_{G} \prod_{j=1}^J f_{j,m_j}( \sigma_j(x) ) \wrt{x} \rabs \\
&\quad\leq C \sum_{m_1, \dots, m_J} \prod_{j} \norm{f_{j,m_j}}_{L^{p_j}(G_j)} \\
&\quad=    C  \prod_{j} \biggl( \sum_{m} \norm{f_{j,m}}_{L^{p_j}(G_j)}\biggr) ,
\end{align*}
and it follows that
\[
\labs \int_{G} \prod_{j=1}^J f_j ( \sigma_j(x) ) \wrt{x} \rabs \\
\leq C \prod_{j} \norm{f_{j}}_{L^{p_j}(G_j)},
\]
as required.
\end{proof}

This theorem allows us to limit our consideration to the elementary LCA groups $G_{\alpha,\beta}$.

\subsection{An application of Weil's integration formula}

Suppose that $H$ is a closed subgroup of $G$ and $\sigma_j(H)$ is a closed subgroup of $G_j$ for all $j$ (in particular, $H$ may be a compact or open subgroup of $G$).
Then each homomorphism $\sigma_j$ from $G$ to $G_j$ induces a homomorphism $\dot\sigma_j$ from $G/H$ to $G_j/\sigma_j(H)$; more precisely,
\[
\dot\sigma_j(xH) = \sigma_j(x)\sigma_j(H).
\]
Further, we may treat integration over $G$ as the result of integration over $H$ followed by integration over $G/H$, and integration over $G_j$ as result of integration over $\sigma_j(H)$ followed by integration over $G_j/\sigma_j(H)$.
This leads us to the following theorem (compare with \cite[Lemma 4.7]{BCCT08}).

\begin{thm}\label{thm:BL-G-subgps-quotients}
Suppose that $(G, \bss,\bsp)$ is a nondegenerate Brascamp--Lieb datum and that $H$ is a closed subgroup of $G$. Then $\sigma_j(H)$ is a closed subgroup of $G_j$ for all $j$, and 
\[
\BL(G, \bss, \bsp)
\leq \BL(H, \bss|_H, \bsp) \, \BL(G/H, \bsds, \bsp) .
\]
Moreover, if $H$ is open, and $\BL(G/H, \bsds, \bsp)) = 1$, then
$\BL(G, \bss, \bsp) = \BL(H, \bss|_H, \bsp)$.
Likewise, if $H$ is compact and $\BL(H, \bss|_H, \bsth) = 1$, then $\BL(G, \bss, \bsth) = \BL(G/H, \bsds, \bsth)$.
\end{thm}

\begin{proof}
According to Weil's formula and Definition \ref{def:BL-ineq-homo}, we may write
\allowdisplaybreaks
\begin{align*}
&\int_{G} \prod_{j=1}^J f_j ( \sigma_j(x) ) \wrt{x} \\
&\quad= \int_{G/H}\int_{H} \prod_{j=1}^J f_j( \sigma_j(xy) ) \wrt{y} \,d\dot x \\
&\quad\leq \BL(H, \bss|_H , \bsp) \int_{G/H} \prod_{j=1}^J
\biggl( \int_{\sigma_j(H)} |f_j( \sigma_j(x) y_j ) |^{p_j} \wrt{y}_j \biggr) ^{1/p_j} \,d\dot x  \\
&\quad\leq \BL(H, \bss|_H, \bsp) \, \BL(G/H, \bsds, \bsp)
\prod_{j=1}^J
\biggl( \int_{G_j/\sigma_j(H)} \int_{\sigma_j(H)} |f_j( x_j y_j )|^{p_j} \wrt{y}_j  \wrt{\dot x}_j \biggr) ^{1/p_j} \\
&\quad= \BL(H, \bss|_H, \bsp) \, \BL(G/H, \bsds, \bsp)
\prod_{j=1}^J \biggl( \int_{G_j} |f_j( x_j )|^{p_j} \wrt{x}_j \biggr) ^{1/p_j} ,
\end{align*}
which leads to the desired inequality.
On the third line, we applied the Brascamp--Lieb inequality with homomorphisms $\sigma_j|_H$ and input functions $y_j \mapsto f_j(\sigma_j(x) y_j)$ on $\sigma_j(H)$, and on the fourth line, we applied the Brascamp--Lieb inequality with homomorphisms $\dot\sigma_j$ and input functions
\[
\dot x_j \mapsto \biggl(\int_{\sigma_j(H)} |f_j(x_j y_j)|^{p_j} \wrt{y}_j\biggr)^{1/p_j}
\]
on $G_j/\sigma_j(H)$.

From Theorems \ref{thm:BL-cpct-qtnt} and \ref{thm:BL-open-sbgp}, we know that
\[
\BL(H, \bss|_H, \bsth)
\leq \BL(G, \bss, \bsth)
\]
when $H$ is open, and
\[
\BL(G/K, \bsds, \bsth)
\leq \BL(G, \bss, \bsth)
\]
when $K$ is compact.
The rest of the proof  follows immediately.
\end{proof} 

When we deal with $\BL(G, \bsG,\bsp)$, we can use absolute values, but not when we deal with $\BLhat(G, \bsG,\bspp)$; thus the proof of a version of Theorem \ref{thm:BL-G-subgps-quotients} for the dual constants will require different arguments.

\subsection{Product groups and homomorphisms}
The next theorem is now evident.

\begin{cor}\label{cor:if-factor}
Suppose that $(G, \bss, \bsp)$ is a nondegenerate Brascamp--Lieb datum, and that we can factorise $G$ and all $\sigma_j$, as follows:
\[
G = G\one \times G\two
\qquad\text{and}\qquad
\sigma_j = \sigma_j\one\otimes\sigma_j\two,
\]
where $G\one$ and $G\two$ are LCA groups and $\sigma_j^{(i)}$ is the restriction of $\sigma_j$ to $G^{(i)}$.
Then
\[
\BL(G, \bss, \bsp)
= \BL(G\one, \bss\one, \bsp)
\BL(G\two, \bss\two, \bsp) .
\]
\end{cor}

\begin{proof}
We have already shown that
\[
\BL(G, \bss, \bsp)
\leq \BL(G\one, \bss\one, \bsp)
\BL(G\two, \bss\two, \bsp) .
\]
The converse inequality follows by testing on products of near-extremals for $G\one$ and near-extremals for $G\two$.
\end{proof}

Unfortunately, there are many examples of homomorphisms of product groups that do not split nicely as in the hypothesis of this theorem.

\section{Previous results and the rank condition}\label{sec:rank-cond}

The Brascamp--Lieb constant has been intensively studied in certain particular settings, and here we extend these results a little, to deal with the four types of groups that make up general LCA groups. The literature on Brascamp--Lieb inequalities also contains some nonabelian examples, for which we refer the reader to \cite{Bramati, Fass, Stovall, Zh24} and the references there.

\subsection{Euclidean spaces}
As we have already mentioned in the introduction, characterisations of the Brascamp--Lieb data for which the Brascamp--Lieb constant is finite were established in
\cite{BCCT08, BCCT10}, following earlier characterisations in \cite{Barthe, CLL04} in the case of rank-one homomorphisms $\sigma_j$.
There are many other fundamental contributions to our understanding of the Brascamp--Lieb constant in the Euclidean setting, such as Lieb's  theorem, allowing the Brascamp--Lieb constant to be computed by testing on centred gaussian inputs $f_j$. Another notable example is in \cite{GGOW}, where an operator scaling algorithm is shown to identify finiteness of the Brascamp--Lieb constant in polynomial time.
We recall from \cite{BCCT08} that, when $G$ and each $G_j$ are vector spaces, $V$ and $V_j$ say, necessary and sufficient conditions for the finiteness of the Brascamp--Lieb constant $\BL(G, \bss, \bsp)$ are the \emph{homogeneity condition}
\[
\dim(V) = \sum_{j} \dim(V_j) / p_j,
\]
and the \emph{rank condition}: for all subspaces $W$ of $V$,
\[
\dim(W) \leq \sum_{j} \dim(\sigma_j(W)) / p_j .
\]
Similar conditions were found for the case where $G$ and all $G_j$ are compact or discrete in \cite{BCCT08}.
In \cite{BJ}, it was observed that group duality links the conditions in these cases.

\subsection{Torsion-free discrete groups}
We assume here that the Haar measure on a discrete group is the counting measure.

If $G$ is discrete and torsion-free, and $\sigma_j: G \to G_j$ is a surjective open homomorphism, then $G_j$ is also discrete, but need not be torsion-free.
The case where $G$ is discrete and finitely generated is analysed in \cite{C13}, following previous work in \cite{CDKSY}.  
When $G$ is also torsion-free, it was shown in \cite{CDKSY} that the Brascamp--Lieb constant is either $1$ or $\infty$, and it is finite (see \cite{BCCT08}) if and only if the rank condition 
\begin{equation}\label{eq:rank-condition-dtf}
  \gamma(H) \leq \sum_{j} \gamma(\sigma_j(H)) /  p_j
\end{equation}
for all subgroups $H$ of $G$, holds.

\begin{lem}\label{lem:discrete-torsion-free-case}
Suppose that $G$ is a torsion-free discrete group and $(G, \bss, \bsp)$ is a nondegenerate Brascamp--Lieb datum.
If the rank condition \eqref{eq:rank-condition-dtf} holds for $G$, then $\BL(G, \bss,\bsp) = \BLhat(G, \bss, \bspp) =1$; otherwise, $\BL(G, \bss,\bsp) = \BLhat(G, \bss, \bspp) = \infty$.
\end{lem}

\begin{proof}
A general torsion-free discrete group $G$ is the inductive limit of finitely generated discrete torsion-free groups $G_\alpha$ by structure theory, and 
$\BL(G, \bss,\bsp) = \lim_{\alpha} \BL(G_\alpha, \bss,\bsp)$ by Theorem \ref{thm:BL-approx}.
The rank condition holds for $G$ if and only if it holds for all $G_\alpha$ by definition.

If the rank condition holds for $G$, then it holds for all $G_\alpha$, so $\BL(G_\alpha, \bss,\bsp) = 1$ for all $\alpha$, and so $\BL(G, \bss,\bsp) = 1$.
Conversely, if $\BL(G, \bss,\bsp) = 1$ then necessarily $\BL(G_\alpha, \bss,\bsp) = 1$ for all $G_\alpha$ and the rank condition holds for all $G_\alpha$ and hence for  $G$.

A similar argument is valid for $\BLhat(G, \bss, \bspp)$.
\end{proof}

\subsection{Compact connected groups}
We assume here that compact groups have total Haar measure $1$.

If $G$ is compact and connected, and $\sigma_j: G \to G_j$ is a surjective open homomorphism, then $G_j$ is also compact and connected.
If $G$ is a torus, that is, $\hat{G}$ is finitely generated, then so are the $G_j$.
The case where $G$ is a torus is analysed in \cite{BJ} (in the subgroup version of the Brascamp--Lieb inequalities), based on Fourier duality and previous work of \cite{CDKSY}.  
According to \cite[Theorem 4.1]{BJ}, when $G$ is a torus, the Brascamp--Lieb constant is either $1$ or $\infty$, and it is finite if and only if 
\begin{equation}\label{eq:dual-rank-condition-Gcc}
  \gamma(\hat{H}) \leq \sum_{j} \gamma(\sigma_j(\hat{H})) /  p_j
\end{equation}
for all subgroups $\hat{H}$ of the discrete torsion-free group $\hat{G}$.

\begin{lem}\label{lem:compact-connected-case}
Suppose that $G$ is a compact connected group and $(G, \bss, \bsp)$ is a nondegenerate Brascamp--Lieb datum.
If the dual rank condition \eqref{eq:dual-rank-condition-Gcc} holds, then $\BL(G, \bss,\bsp) = \BLhat(G, \bss, \bspp) =1$; otherwise, $\BL(G, \bss,\bsp) = \BLhat(G, \bss, \bspp) = \infty$.
\end{lem}

\begin{proof}
The general compact connected group $G$ is the projective limit of tori $G_\beta$ by structure theory, and 
$\BL(G, \bss,\bsp) = \lim_{\beta} \BL(G_\beta, \bss,\bsp)$ by Theorem \ref{thm:BL-approx}.
The dual rank condition holds for $G$ if and only if it holds for all $G_\beta$ by definition.

If the dual rank condition holds for $G$, then it holds for all $G_\beta$, so $\BL(G_\beta, \bss,\bsp) = 1$ for all $\beta$, and so $\BL(G, \bss,\bsp) = 1$.
Conversely, if $\BL(G, \bss,\bsp) = 1$ then necessarily $\BL(G_\beta, \bss,\bsp) = 1$ for all $G_\beta$ and the dual rank condition holds for all $G_\beta$ and hence for  $G$.

A similar argument is valid for $\BLhat(G, \bss, \bspp)$.
\end{proof}

\subsection{Totally disconnected groups with totally disconnected duals}
A general totally disconnected group $G$ with totally disconnected dual is a limit of finite groups $G_{\alpha,\beta}$, and the situation for finite groups is also known, thanks to \cite{CDKSY} and \cite{BJ}.
The next result follows from Theorem \ref{thm:BL-approx} and \cite[Theorems 4.1 and 4.2]{BJ} in the same way that Lemmas \ref{lem:discrete-torsion-free-case} and \ref{lem:compact-connected-case} do.

\begin{lem}\label{lem:tdG-with-TDGhat}
Suppose that $G$ is a totally disconnected group with totally disconnected dual and $(G, \bss, \bsp)$ is a nondegenerate Brascamp--Lieb datum.
Then
\[
\BL(G, \bss, \bsp) 
= \BLhat(G, \bss, \bspp) 
= \sup_{H \in \calO(G)} \frac{\norm{1_H}_1} {\prod_{j} \norm{1_{\sigma_j(H)}}_{p_j}}
= \sup_{H \in \calO(G^\perp)} \frac{\norm{1_H}_1} {\prod_{j} \norm{1_{\sigma_j(H)}}_{p_j'}}  .
\]
\end{lem}
We leave the verification of this to the reader.
We point out explicitly that in the approximation process, the Haar measure of $G$ determines that on the approximating finite groups $G_{\alpha,\beta}$, by the requirement that the subgroups and quotients have the standard open subgroup and compact quotient measures.

\subsection{The rank and dual rank conditions for a general LCA group}
We now  start putting everything together to deal with general LCA groups.

Every compactly generated LCA group $G$ has a \emph{growth index} $\gamma(G)$, such that $|U^n| \lesssim_U n^{\gamma(G)}$ for all positive integers $n$ and all relatively compact subsets of $G$, and $|U^n| \gtrsim_U n^{\gamma(G)}$ for some such $U$.  
A general compactly generated LCA group $G$ is of the form $\R^a \times \T^b \times K$, where $K$ is a compact subgroup, and the growth index $\gamma(G)$ of $G$ is then $a+c$.

We call a set $U$ of the form $[-r,r]^a \times [-r,r]^c \times K$ (for some fixed but large $r$) a product interval set.
If $U$ is a product interval set, then it is easy to check that
\[
\norm{1_{U^n}}_{L^p(G)} \eqsim_U \norm{\hat 1_{U^n}}_{L^{p'}(G)} \eqsim_U n^{\gamma(G)/p} .
\]

Given a nondegenerate datum $(G, \bss, \bsp)$ and a compactly generated closed subgroup $H$ of $G$, take a relatively compact open subset $U$ of $G$ and a relatively compact open subset $V$ of $H$ that generates $H$ and satisfies $|V^n| \eqsim_V n^{\gamma(H)}$; then the sets $\sigma_j(U)$ and $\sigma_j(V)$ are relatively compact in $\sigma_j(H)\subseteq G_j$.
We take product interval sets $W_j$ in each $G_j$ such that $\sigma_j(U) \subseteq W_j$ and $\sigma_j(V) \subseteq W_j$ (this amounts to a choice of the parameter $r$),  and then
\[
n^{\gamma(H)}
\lesssim \int_{G} 1_{UV^n}(x) \,dx
\leq \int_{G} \prod_{j}1_{\sigma_j(U)\sigma_j(V)^n}(x) \,dx
\leq \int_{G} \prod_{j}1_{\sigma_j(W)^{n+1}}(x) \,dx ,
\]
while
\[
\prod_j \norm{1_{W^{n+1}}}_{p_j}
= \prod_j | W^{n+1} |^{1/p_j}
\lesssim_U \prod_j (n+1)^{\gamma(\sigma(H)) /p_j}
\lesssim n^{\sum_j \gamma\sigma_j(H) )/p_j} ,
\]
and a similar estimate holds when we consider $\prod_j \norm{\hat 1_{W^n}}_{p'_j}$.

It follows that a necessary condition for a Brascamp--Lieb inequality or a dual Brascamp--Lieb inequality to hold is the rank condition:
\begin{equation}\label{eq:rank-inequality}
\gamma(H) \leq \sum_j \gamma(\sigma_j(H))/p_j
\end{equation}
for all compactly generated closed subgroups $H$ of $G$.
A duality argument shows that another necessary condition for a Brascamp--Lieb inequality or a dual Brascamp--Lieb inequality to hold is the dual rank condition:
\begin{equation}\label{eq:dual-rank-inequality}
\gamma(\hat{H}) \leq \sum_j \gamma(\sigma_j(\hat{H}))/p'_j
\end{equation}
for all compactly generated closed subgroups $\hat{H}$ of $\hat{G}$.

As remarked earlier, the rank condition is sufficient for torsion-free discrete groups, and the dual condition is sufficient for compact connected groups.
For vector spaces, which are self-dual, condition \eqref{eq:rank-inequality} and its dual version imply the homogeneity condition of BCCT.

\section{The general case}\label{sec:gen}

In this section, we consider the general case.
Recall that, for a general LCA group $G$, we write $G^{1} := G_{b} \cap G_{c}$; $G^{2} := G_{c}/(G_{b} \cap G_{c})$; $G^{3} := (G_{b} G_{c}) /G_{c}$; and $G^{4} := G/(G_{b} G_{c})$.
If $\sigma_j: G \to G_j$ is a homomorphism of LCA groups, then by a combination of restriction and factoring, $\sigma_j$ induces homomorphisms $\sigma^{1}_j: G^{1} \to G^{1}_j$, $\sigma^{2}_j: G^{2} \to G^{2}_j$, $\sigma^{3}_j: G^{3} \to G^{3}_j$, and $\sigma^{4}_j: G^{4} \to G^{4}_j$.
The vector homomorphism $\bss$ is then proper if and only if all the vector homomorphisms $\bss^1$, $\bss^2$, $\bss^3$ and $\bss^4$ are proper, by Proposition \ref{prop:new-homos-from-old}.
We normalise the Haar measures on $G^1$ and on $G^1_j$ to have total mass $1$ and the Haar measures on $G^4$ and on $G^4_j$ to be counting measures.

We start with $G^1$ and $G^4$.

\begin{lem}\label{lem:Gamma}
Let $(G, \bss, \bsp)$ be a nondegenerate Brascamp--Lieb datum, such that $\BL(G, \bss, \bsp)$ or $\BLhat(G, \bss, \bspp)$ is finite.
Then  
\[
\BL(G^4, \bssf, \bsp) = \BLhat(G^4, \bssf, \bspp) = 1.
\] 
Likewise,
\[
\BL(G^1, \bsso, \bsp) = \BLhat(G^1, \bsso, \bspp) = 1.
\] 
\end{lem}

\begin{proof}
Let $\pi$ be the canonical projection from $G$ to the group $G^4$, which is discrete and torsion free, and is equipped with the counting measure.
Every compactly generated subgroup $\dot\Gamma$ of $G^4$ is finitely generated: let $\dot{E}$ be a minimal generating set for $\dot\Gamma$ in $G^4$, and choose any subset $E$ of $G$ that is mapped by $\pi$ bijectively to $\dot{E}$.
Then $\Gamma$, the group generated by $E$, which we also equip with counting measure, is a discrete subgroup of $G$, and $\pi:\Gamma \to \dot\Gamma$ is an isomorphism.  

Since $\BL(G, \bss, \bsp)<\infty$ or $\BLhat(G, \bss, \bspp)<\infty$, the arguments of Section \ref{sec:rank-cond} imply that $\gamma(\Gamma) \leq \sum_j \gamma(\sigma_j(\Gamma))/p_j$, whence
\[
\gamma(\dot\Gamma) \leq \sum_j \gamma(\sigma_j(\dot\Gamma))/p_j,
\]
and since $\dot\Gamma$ is an arbitrary compactly generated subgroup of $G^4$, the rank condition \eqref{eq:rank-inequality} holds for $G^4$.
By Lemma \ref{lem:discrete-torsion-free-case}, $\BL(G^4, \bssf, \bsp) = \BLhat(G^4, \bssf, \bspp) = 1$.

An analogous argument involving the dual group of $G^1$, and using Lemma \ref{lem:compact-connected-case}, completes the proof.
\end{proof}

\begin{proof}[Proof of Main Theorem]
If the product
\[
\BL(G^1, \bss^1, \bsp)  \BL(G^2, \bss^2, \bsp) \BL(G^3, \bss^3, \bsp) \BL(G^4, \bss^4, \bsp)
\] 
is finite, so is $\BL(G, \bss, \bsp)$, by a repeated application of Theorem \ref{thm:BL-G-subgps-quotients}, and then 
Lemma \ref{lem:Gamma} implies that $\BL(G^4, \bssf, \bsp) = \BL(G^1, \bsso, \bsp) = 1$.  
Two applications of Theorem \ref{thm:BL-G-subgps-quotients} now show that 
\[
\BL(G, \bss, \bsp) = \BL(G_{b}G_{c}/(G_{b} \cap G_{c}), \bsds,\bsp),
\]
where $\dot\sigma_j$ is the induced homomorphism from $G_{b}G_{c}/(G_{b} \cap G_{c})$.

As noted in Section 2, $G_{b}G_{c}/(G_{b} \cap G_{c})$ is the direct product of a vector group $G_{c}/G_{b} \cap G_{c}$ and a totally disconnected group of compact elements $G_{b}/G_{b} \cap G_{c}$, and any homomorphism of such a group into another group of the same type must send the connected component into the connected component of the image group and bounded elements into bounded elements, and so preserve the direct product structure.
By Corollary \ref{cor:if-factor},
\begin{align*}
\BL(G_{b}G_{c}/(G_{b} \cap G_{c}), \bsds,\bsp) 
&= \BL(G_{c}/(G_{b} \cap G_{c}), \bsds,\bsp) \BL(G_{b}/(G_{b} \cap G_{c}), \bsds,\bsp) \\
&= \BL(G^2, \bsstw,\bsp) \BL(G^3, \bssth,\bsp).
\end{align*}

To conclude the proof, if 
\[
\BL(G^1, \bss^1, \bsp)  \BL(G^2, \bss^2, \bsp) \BL(G^3, \bss^3, \bsp) \BL(G^4, \bss^4, \bsp)
\] 
is infinite, then at least one of the factors is infinite.
If, say, $\BL(G^4, \bss^4, \bsp)$ is infinite, then there is a finite generated subgroup $\dot\Gamma$ of $G^4$ for which the rank condition fails; hence there is a corresponding compactly generated discrete subgroup $\Gamma$ of $G$ for which the rank condition fails; finally, $G$ fails the rank condition, and $\BL(G, \bss, \bsp)$ is infinite.
A similar argument shows that $\BL(G, \bss, \bsp)$ is infinite if $\BL(G^1, \bss^1, \bsp)$ is infinite, and with a little more effort, we may  show that $\BL(G, \bss, \bsp)$ is infinite if $\BL(G^2, \bss^2, \bsp)$ or $\BL(G^3, \bss^3, \bsp)$ is infinite.
\end{proof}

\section{Remarks}\label{sec:rem}
In this section, we point out various reductions that are available to someone who might wish to compute a Brascamp--Lieb constant in the generality of LCA groups.

\subsection{When $p_j$ is $\infty$ or $1$}
We shall show that when one or more of the $p_j$ is equal to $\infty$ or $1$, then it suffices to consider ``smaller groups''.

\begin{thm}
Suppose that $1 \leq j \leq J$, and denote by $\tilde{\bss}$ and $\tilde{\bsp}$ the collection of homomorphisms $\sigma_1, \dots,\sigma_{k-1}, \sigma_{k+1}, \dots, \sigma_J$ and the collection of indices $p_1, \dots, p_{k-1},p_{k+1}, \dots, p_J$ with $\sigma_k$ omitted  and with $p_k$ omitted.
If $p_k = \infty$, then
\[
\BL(G, \bss, \bsp)
= \BL(G, \tilde{\bss}, \tilde{\bsp})
\qquad\text{and}\qquad
\BLhat(G, \bss, \bspp)
= \BLhat(G, \tilde{\bss}, \tilde{\bsp}').
\]
\end{thm}

\begin{proof}
The inequalities 
\[
\BL(G, \tilde{\bss}, \tilde{\bsp}) \leq \BL(G, \bss, \bsp)
\qquad\text{and}\qquad
\BLhat(G, \tilde{\bss}, \tilde{\bsp}') \leq \BLhat(G, \bss, \bspp)
\] 
are easy.  
Equality relies on the existence of sequences $(u_n)$ in $S(G_k)$ that converge to $1$ on arbitrarily large compact sets and satisfy $\norm{u_n}_{L^\infty(G_k)} \leq 1$ and $\norm{\hat u_n}_{L^1(\hat{G}_k)} \leq 1$.
\end{proof}

This justifies omission of the case where some $p_j$ are infinite.

\begin{thm}
Suppose that $1 \leq k \leq J$ and $p_k = 1$. 
Denote $\ker(\sigma_k)$ by $N$, denote by $\bsr$ the collection of restricted homomorphisms $\sigma_1|_{N}, \dots,\sigma_{k-1}|_{N}, \sigma_{k+1}|_{N}, \dots, \sigma_J|_{N}$ with $\sigma_k$ omitted, and denote by $\bsq$ the collection of indices $p_1, \dots, p_{k-1},p_{k+1}, \dots, p_J$  with $p_k$ omitted.
Then
\[
\BL(G, \bss, \bsp)
= \BL(N, \bsr, \bsq).
\]
\end{thm}

\begin{proof}
Suppose that $\BL(G, \bss, \bsp)$ is finite, and write $C$ for $\BL(G, \bss, \bsp)$.
Take $z \in G$, and let $(f^\alpha_k)_\alpha$ be an approximation to the delta distribution at $\sigma_k(z)$ in $G_k$.
Since $\bss$ is proper, $\bigcap_j \sigma_j^{-1}(\supp(f_j))$ is a compact subset of $G$.
Then Weyl's integral formula \eqref{eq:Weils-formula} shows that
\begin{equation}\label{eq:L1-reduction-1}
\int_{G} \prod_{j} f_j(\sigma_j(x)) \wrt{x}
= \int_{G/N} f_k(\dot x) \int_{N} \prod_{j\neq k} f_j(\sigma_j(xy)) \wrt{y} \wrt{\dot{x}}
\to \int_{N} \prod_{j\neq k} f_j(\sigma_j(z) \rho_j(y)) \wrt{y} ,
\end{equation}
whence
\begin{equation}\label{eq:L1-reduction-2}
\bigglabs \int_{N} \prod_{j\neq k} f_j(\sigma_j(z)\rho_j(y)) \wrt{y} \biggrabs
\leq C \prod_{j\neq k} \norm{ f_j }_{L^{p_j}(G_j)} 
\qquad\forall f_j \in S(G_j).
\end{equation}
This holds uniformly for $z \in G$, and in particular when $z = e$, and so 
\begin{equation}\label{eq:L1-reduction-3}
\bigglabs \int_{N} \prod_{j\neq k} f_j(\rho_j(y)) \wrt{y} \biggrabs
\leq C \prod_{j\neq k} \norm{ f_j }_{L^{p_j}(G_j)} 
\qquad\forall f_j \in S(G_j),
\end{equation}
which implies that $\BL(N, \bsr, \bsq) \leq C = \BL(G, \bss, \bsp)$.

Conversely, suppose that $\BL(N, \bsr, \bsq)$ is finite and write $C$ for $\BL(N, \bsr, \bsq)$.
Then \eqref{eq:L1-reduction-3} holds by definition, so \eqref{eq:L1-reduction-2} holds, by the translation-invariance of Haar measure, and then
\[
\begin{aligned}
\bigglabs\int_{G} \prod_{j} f_j(\sigma_j(x)) \wrt{x} \biggrabs
&\leq \int_{G/N} |f_k(\dot x)| \bigglabs \int_{N} \prod_{j\neq k} f_j(\sigma_j(xy)) \wrt{y} \biggrabs \wrt{\dot{x}} \\
&\leq C \norm{f_k}_{L^1(G_k)} \prod_{j\neq k} \norm{ f_j }_{L^{p_j}(G_j)},
\end{aligned}
\]
and so $\BL(G,\bss,\bsp) \leq C = \BL(N,\bsr,\bsq)$.
\end{proof}
If $\BL(G,\bss,\bsp)$ is finite, then so is $\BL(N,\bsr,\bsq)$, and hence $\bsr$ is proper.
This implies that $\rho_j(N)$ must be an open subgroup of $G_j$.

This result justifies omission of the case where some $p_j$ are equal to $1$.

\subsection{Transversality}
We say that $\bss$ is \emph{transversal} if $\bigcap_{j \neq k} \ker(\sigma_j) = \{e\}$ when $1 \leq k \leq J$.
This condition implies that, given  $x_1 \in G_1$, \dots, $x_{j-1} \in G_{j-1}$, $x_{j+1} \in G_{j+1}$,  \dots, and $x_n \in G_n$, there is exactly one $x_j \in G_j$ such that $(x_1, \dots, x_n) \in \bss(G)$, that is, that $\bss(G)$ may be viewed as a graph over  $G_1\times \dots \times G_{j-1} \times G_{j+1}\times  \dots \times G_n$.

The next result shows that we may assume transversality.

\begin{thm}\label{thm:BL-non-transverse}
Suppose that $(G, \bss, \bsp)$ is a Brascamp-Lieb datum, and that $N$ is a closed subgroup of $G$ such that $\sigma_k(N)$ is closed in $G_k$ and $\sigma_j(N) = \{e \}$ when $j \neq k$.
If $N$ is non\-compact and $p_k \neq 1$, then 
\[
\BL(G, \bss,\bsp) = \BLhat(G, \bss,\bsp) = \infty.
\]
Otherwise, define $\dot\sigma_k: G/N \to G_k/\sigma(N)$ by $\dot\sigma_k(xN) = \sigma_k(x) \sigma_k(N)$ and $\dot\sigma_j: G/N \to G_j$ by $\dot\sigma_j(xN) = \sigma_j(x)$ when $j \neq k$.
Then 
\[
\BL(G/N, \bsds,\bsp) = \BL(G, \bss,\bsp)
\quad\text{and}\quad
\BLhat(G/N, \bsds,\bsp) = \BLhat(G, \bss,\bsp) .
\]
\end{thm}

\begin{proof}
This follows from the equality 
\begin{equation}\label{eq:quot}
\int_{G} \prod_{j} f_j (\sigma_j(x)) \wrt{x} = \int_{G/N} \prod_{j} \dot{f}_j (\sigma_j(\dot{x})) \wrt{\dot{x}} ,
\end{equation}
where $\dot{f}_j = f_j$ when $j \neq k$ while 
\[
\dot{f}_k(\dot{y}) = \int_{\sigma_k(N)} f_k(yz) \wrt{z} = \int_{N} f_k(y\sigma(z)) \wrt{z}
\] 
for all $y \in G_k$. 
If $N$ is noncompact and $p_k > 1$, we can arrange that $\dot{f}_k$ is infinite on a set of positive measure, and then \eqref{eq:quot} in infinite while $\norm{f_k}_{L^{p_k}(G_k)}$ is finite; otherwise,
\[
\bignorm{\dot{f}_k }_{L^{p_k}(G_k/\sigma_k(N))} \leq \bignorm{ f_k }_{L^{p_k}(G_k)}, 
\]
with equality attained when $f_k$ is of the form $\dot{f}_k(y\sigma(N)) g(y)$, where $\int_{\sigma(N)} g(yz) \wrt{z} = 1$.
Similar inequalities hold for the Fourier transforms.
\end{proof}


\end{document}